\documentclass[12pt]{article}
\usepackage{epsfig}
\usepackage{amsmath}
\usepackage{amsthm}
\usepackage{amssymb}
\usepackage{graphicx}
\usepackage{float}
\usepackage{epstopdf}
\usepackage{soul}
\usepackage{subfigure}
\usepackage{caption}
\usepackage[utf8]{inputenc}
\usepackage{enumerate}
\usepackage[ruled,vlined]{algorithm2e}
\usepackage{cite}

\newtheorem{remark}{Remark}
\newtheorem{corollary}{Corollary}
\newtheorem{lemma}{Lemma}
\newtheorem{theorem}{Theorem}
\newtheorem{definition}{Definition}

\newcommand{\set}[1]{\left\{#1\right\}}

\begin{document}

\title{\bf Experimental Design : Optimizing Quantities of Interest to Reliably Reduce the Uncertainty in Model Input Parameters}
\author{Scott Walsh\\{Advisor: Troy Butler}}
\date{}
\maketitle
\tableofcontents



\newcommand\matlab{{\sc matlab}}
\newcommand{\goto}{\rightarrow}
\newcommand{\bigo}{{\mathcal O}}
\newcommand{\half}{\frac{1}{2}}
\newcommand\reals{{{\rm l} \kern -.15em {\rm R} }}
\newcommand\complex{{\raisebox{.043ex}{\rule{0.07em}{1.56ex}} \hskip -.35em {\rm C}}}


\newenvironment{mat}{\left[\begin{array}{ccccccccccccccc}}{\end{array}\right]}
\newcommand\bcm{\begin{mat}}
\newcommand\ecm{\end{mat}}

\newenvironment{rmat}{\left[\begin{array}{rrrrrrrrrrrrr}}{\end{array}\right]}
\newcommand\brm{\begin{rmat}}
\newcommand\erm{\end{rmat}}

\newenvironment{choices}{\left\{ \begin{array}{ll}}{\end{array}\right.}
\newcommand\when{&\text{if~}}
\newcommand\otherwise{&\text{otherwise}}

\newcommand{\eql}{\begin{equation}\label}
\newcommand{\eqn}[1]{(\ref{#1})}

\newcommand\unp{U^{n+1}}
\newcommand\unm{U^{n-1}}

\newcommand{\react}[1]{\stackrel{K_{#1}}{\rightarrow}}
\newcommand{\reactb}[2]{\stackrel{K_{#1}}{~\stackrel{\rightleftharpoons}
   {\scriptstyle K_{#2}}}~}


\renewcommand{\theenumi}{\alph{enumi}}
\renewcommand{\labelenumi}{(\theenumi)}
\newcommand{\supp}{\mathop{\mathrm{supp}}}

\pagebreak

\vskip 10pt
\section{\bf Introduction}
\vskip 10pt

With an abundance of computing resources available today mathematical models are capable of simulating extraordinarily complex physical systems.  These complex models in turn induce a set of model input parameters, many of which are subject to uncertainty.  For example, the diffusion coefficient of a model simulating the heating of a thin plate by some external source is subject to uncertainty, possibly due to the manufacturing imperfections of the thin plate.  The Manning's n parameters of a coastal storm surge model may initially be determined by extensive exploration of the coastal boundary, however each passing storm has the potential to significantly alter regions of this coastal boundary introducing uncertainty into these parameters.  Our goal is to reduce the uncertainty in these model input parameters, therefore improving the predictive capabilities of the model.

One approach to reducing the uncertainty in these model input parameters is to gather quantities of interest (QoI) of the physical system and use this data to inform us about the parameters that may have produced this data.  This data may be a temperature measurement at a point in space time or a maximum water elevation at a given location in the physical domain.  As our model input parameters are uncertain, so is the data we gather.  The gathered data is uncertain for a multitude of reasons, e.g., measurement instruments have finite precision, instrument locations in space time are subject to uncertainty.  This uncertainty in each QoI gathered may also be a function of the region of the parameter space that produced this physical reality.  Different observed solutions (corresponding to different regions of model input parameters) may produce physical processes that are not being modeled that pollute measurements even more.  For example, some regions of a parameter space of a coastal storm surge model may produce large and chaotic waves that pollute measurements.

This uncertainty in the QoI produces a set valued inverse problem.  In \cite{BE1, BE3} the foundation for a measure theoretic approach to solving this set valued inverse problem is carefully developed.  With this foundation built on well understood mathematical concepts, the uniqueness of such set valued inverse solutions is proved, under suitable conditions.  However, each solution is dependent on the data used in the inverse problem, i.e., different sets of data produce different inverse solutions whose geometric properties may vary substantially.  In this paper we develop an approach for determining an optimal choice of QoI to produce a {\em precise} and {\em accurate} inverse solution, i.e., reduce the uncertainty in our model parameters.

\vskip 10pt

In Section \ref{Sec:Three_Levels} we establish the foundation of the measure theoretic framework for solving the inverse problem.  In Section \ref{Sec:Numerical_Approximation} we consider the numerical algorithm that simultaneously approximates the support of the probability density solving the stochastic inverse problem.  We use simple linear maps to demonstrate the impact the geometry of the inverse density has on the numerical approximation of this support.  In Section \ref{Sec:Precision_Accuracy} we quantify the two guiding geometric properties of the support, the {\em $\mu_\Lambda$-measure} and the {\it skewness}, which represent, in a sense, the precision and accuracy, respectively, of the computed inverse density.  In Section \ref{Sec:Optimizing} we address the need to optimize the minimization of both the $\mu_\Lambda$-measure and skewness.  With this multi-criteria optimization problem defined, we use simple linear and nonlinear maps to develop an intuition about solutions to this problem.  In Section \ref{Sec:Map_Defined}, we demonstrate the impact of determining optimal sets of QoI on solutions to stochastic inverse problems involving physics based models.  
In Section \ref{Sec:Conclusion} we provide concluding remarks and discuss promising future research directions.

\section{\bf A Measure-Theoretic Stochastic Inverse Problem}\label{Sec:Three_Levels}

\subsection{A Mathematical Formulation}

We provide a taxonomy of three forward problems with increasing levels of uncertainty and the direct inverses of these problems along with a brief summary of the corresponding solutions.
For a more thorough description and analysis of these problems and solutions, see \cite{BE1, BE3}. 
Let $\Lambda$ denote the set of all uncertain input parameters for some physics-based model and $\mathcal{D}$ denote the set of possible output data defined as the range of the QoI map, $Q:\Lambda\to\mathcal{D}$, which we assume is piecewise smooth.

\vskip 5pt
{\it\noindent Level 1}

The first forward problem is the simplest in predictive science which is to evaluate the map $Q$ for a fixed $\lambda\in\Lambda$ to determine the corresponding output datum $Q(\lambda)=q\in\mathcal{D}$. 
In other words, once the inputs of a model are specified, the simplest forward problem is to solve the model in order to predict the output datum. 
Generally, solutions to this problem require developing a computational model to determine approximate solutions to the physics-based model and to then apply a functional to the numerical solution to obtain the value of $q$. 
With the exception of possible discretization errors in numerically evaluating $Q(\lambda)$, there is no uncertainty in the forward problem since $Q(\lambda)$ is well-defined. 

The corresponding inverse problem is to determine the possible parameters $\lambda\in\Lambda$ that produce a particular value of $q\in\mathcal{D}$.  
Oftentimes $Q^{-1}(q)$ defines a set of values in $\Lambda$ either due to non-linearities in the map $Q$ and/or the dimension of $\Lambda$ being greater than the dimension of $\mathcal{D}$. 
Thus, the simplest inverse problem often has uncertainty as to the particular value of $\lambda\in\Lambda$ that produced a fixed output datum. 
However, there is no uncertainty about which set-valued inverse produced the fixed output datum.
When $\dim(\Lambda)=2$ and $\dim(\mathcal{D})=1$, a set-valued inverse, $Q^{-1}(q)$, defines a contour in $\Lambda$ that we are familiar with from contour maps, see the left plot in Figure~\ref{fig:contour_illustration}. 
We restrict focus to problems where $\Lambda\subset\reals^n$, $\mathcal{D}\subset\reals^m$, and  $m<n$ (or occasionally $m\leq n$), which corresponds to the common case when there are more input parameters than observable outputs.
We refer to the set-valued inverses of the map $Q$ as {\em generalized contours}.
Generalized contours may be defined by the union of piecewise-defined manifolds in $\Lambda$, and the local dimension of the manifolds depends upon $m$, $n$, and the rank of the Jacobian of $Q$.
\begin{definition}
The component maps of the $m$-dimensional piecewise-smooth vector-valued map $Q(\lambda)$ are {\bf geometrically distinct (GD)} if the Jacobian of $Q$ has full rank at every point in $\Lambda$.
When the component maps are GD, we say that $Q$ is GD. 
\end{definition}
When $m<n$ and $Q$ is GD, the generalized contours exist as piecewise-defined $(n-m)$-dimensional manifolds in $\Lambda$.
In \cite{BE1}, a method for explicit pointwise approximation of generalized contours for scalar multivariate $Q$ maps is provided.
In \cite{BE3}, this was extended for the more general case of vector-valued multivariate $Q$ maps. 
While explicit approximations are required for numerical solutions to this ``simplest'' inverse problem, it is thankfully not necessary for solutions to the stochastic inverse problem which implicitly exploit the geometric structure of the generalized contour map. 

\vskip 5pt
{\noindent\it Level 2}

Following solutions to the first forward problem, it is typical to define a second forward problem from the general class of problems that fall under the category of deterministic sensitivity analysis.  
Some specific deterministic sensitivity analysis problems arise in the context of dynamical systems analysis, where we study how sets of initial conditions vary over time, or in perturbation analysis, where we study how small changes in model inputs affect model outputs.
These types of forward problems can often be written mathematically as analyzing $Q(A)\subset\mathcal{D}$ given a set $A\subset\Lambda$. 
In other words, there is generally uncertainty as to the precise value $\lambda$ takes in a set $A\subset\Lambda$ and subsequently there is uncertainty as to the particular datum to predict in the set $Q(A)$. 
Answering such questions often requires, at a minimum, specification of metrics on both $\Lambda$ and $\mathcal{D}$ so that we can describe distances between points.
Thus, in the formulation of a forward {\em deterministic} sensitivity analysis problem, we often must specify metrics on both $\Lambda$ and $\mathcal{D}$. 
By transfinite induction using sets in the metric topologies, we can construct the Borel $\sigma$-algebras, $\mathcal{B}_\Lambda$ and $\mathcal{B}_{\mathcal{D}}$, on $\Lambda$ and $\mathcal{D}$, respectively.
The metrics can be used to define outer measures on the measurable spaces $(\Lambda,\mathcal{B}_\Lambda)$ and $(\mathcal{D},\mathcal{B}_{\mathcal{D}})$.
Then, Carath\'{e}odory's theorem can be used to define the Hausdorff measures, which we refer to as the ``volume'' measures on $(\Lambda,\mathcal{B}_\Lambda)$ and $(\mathcal{D},\mathcal{B}_{\mathcal{D}})$, and denote by $\mu_\Lambda$ and $\mu_{\mathcal{D}}$, respectively. 
Note that if the metrics are norm-induced on $\Lambda\subset\mathbb{R}^n$ and $\mathcal{D}\subset\mathbb{R}^m$, then up to a scaling constant, $\mu_\Lambda$ and $\mu_{\mathcal{D}}$ are identical to the standard Lebesgue measures on these spaces. 
We observe that construction of the measure spaces $(\Lambda,\mathcal{B}_\Lambda,\mu_\Lambda)$ and $(\mathcal{D},\mathcal{B}_{\mathcal{D}},\mu_{\mathcal{D}})$ follows directly from the specification of metrics on the sets $\Lambda$ and $\mathcal{D}$. 
Therefore, we refer to the specification of metrics as the {\em minimal assumptions} required to obtain the measure-theoretic framework in which the remaining problems in the taxonomy are formulated. 

It is worth noting that a metric need not be specified on $\mathcal{D}$ as long as $\mathcal{D}$ is a topological space.
In this case, we can still obtain a Borel $\sigma$-algebra on $\mathcal{D}$ as before, and define the volume measure $\mu_{\mathcal{D}}$ on $(\mathcal{D},\mathcal{B}_{\mathcal{D}})$ in terms of the induced ``push forward'' measure \begin{equation}
    \mu_{\mathcal{D}}(A) = \mu_{\Lambda}(Q^{-1}(A)), \ A\in\mathcal{B}_{\mathcal{D}},
\end{equation} 
where the measurability of $Q$, which follows from the assumed smoothness properties of $Q$, implies $Q^{-1}(A)\in\mathcal{B}_\Lambda$. 

The corresponding {\em deterministic} inverse sensitivity analysis problem is to analyze $Q^{-1}(A)$ for some $A\in\mathcal{B}_{\mathcal{D}}$.  
While $Q^{-1}(A)\in\mathcal{B}_\Lambda$, the practical computation of $Q^{-1}(A)$ is complicated by the fact that $Q^{-1}$ does not map individual points in $\mathcal{D}$ to individual points in $\Lambda$.
As described above, assuming $m<n$, $Q^{-1}$ maps a point in $\mathcal{D}$ to an $(n-m)$-dimensional manifold embedded in $\Lambda$.
Thus, $Q^{-1}(A)$ defines a {\em generalized contour event}, or just contour event, belonging to an {\em induced contour $\sigma$-algebra} $\mathcal{C}_\Lambda\subset\mathcal{B}_\Lambda$ (where the inclusion is often proper). 
In other words, solutions to the corresponding inverse problem can be described in the measurable space $(\Lambda,\mathcal{C}_\Lambda)$ and the volume measure $\mu_\Lambda$ can be used to provide quantitative assessments of solutions to this inverse sensitivity analysis problem. 

As shown below in describing solutions to the stochastic inverse problem, it is useful to consider solutions to this deterministic inverse sensitivity problem in a space of equivalence classes.
Specifically, we may use the generalized contours to define an equivalence class representation of $\Lambda$ where two points are considered equivalent if they lie on the same generalized contour.  
We let $\mathcal{L}$ denote the space of such equivalence classes and let $\pi_{\mathcal{L}}:\Lambda\to\mathcal{L}$ denote the projection map where $\pi_{\mathcal{L}}(\lambda)=\ell\in\mathcal{L}$ defines the equivalence class corresponding to a particular $\lambda$ and $\pi_{\mathcal{L}}^{-1}(\ell)=C_\ell$ is the generalized contour in $\Lambda$ corresponding to the point $\ell\in\mathcal{L}$.
It is possible to explicitly represent $\mathcal{L}$ in $\Lambda$ by choosing a specific representative element from each equivalence class.
As described in \cite{BE1, BE3}, such a representation of $\mathcal{L}$ can be constructed by piecewise $m$-dimensional manifolds that {\em index} the $(n-m)$-dimensional generalized contours.
From a non-technical perspective, this is like defining a particular hiking path using a contour map where each elevation is transversed once and only once, see the middle plot in Figure~\ref{fig:contour_illustration}. 
We refer to any such indexing manifold as a {\em transverse parameterization}.
Given a particular indexing manifold representing $\mathcal{L}$, $Q$ defines a bijection between $\mathcal{L}$ and $\mathcal{D}$.
The measure space $(\mathcal{L},\mathcal{B}_{\mathcal{L}},\mu_{\mathcal{L}})$ can be defined as an induced space using the bijection $Q$ and $(\mathcal{D},\mathcal{B}_{\mathcal{D}},\mu_{\mathcal{D}})$. 
Then, solutions to the deterministic inverse sensitivity analysis problem can be described and analyzed as measurable sets of points in $\mathcal{L}$ instead of measurable generalized contour events in $\Lambda$. 
While the measure space $(\mathcal{L},\mathcal{B}_{\mathcal{L}},\mu_{\mathcal{L}})$ is useful in describing both the theoretical solutions and computational algorithms approximating solutions to the stochastic inverse problem defined below, it is not necessary to explicitly construct a transverse parameterization. 
\begin{figure}[htbp]
  \centering
 \hbox{ \includegraphics[width=0.33\textwidth]{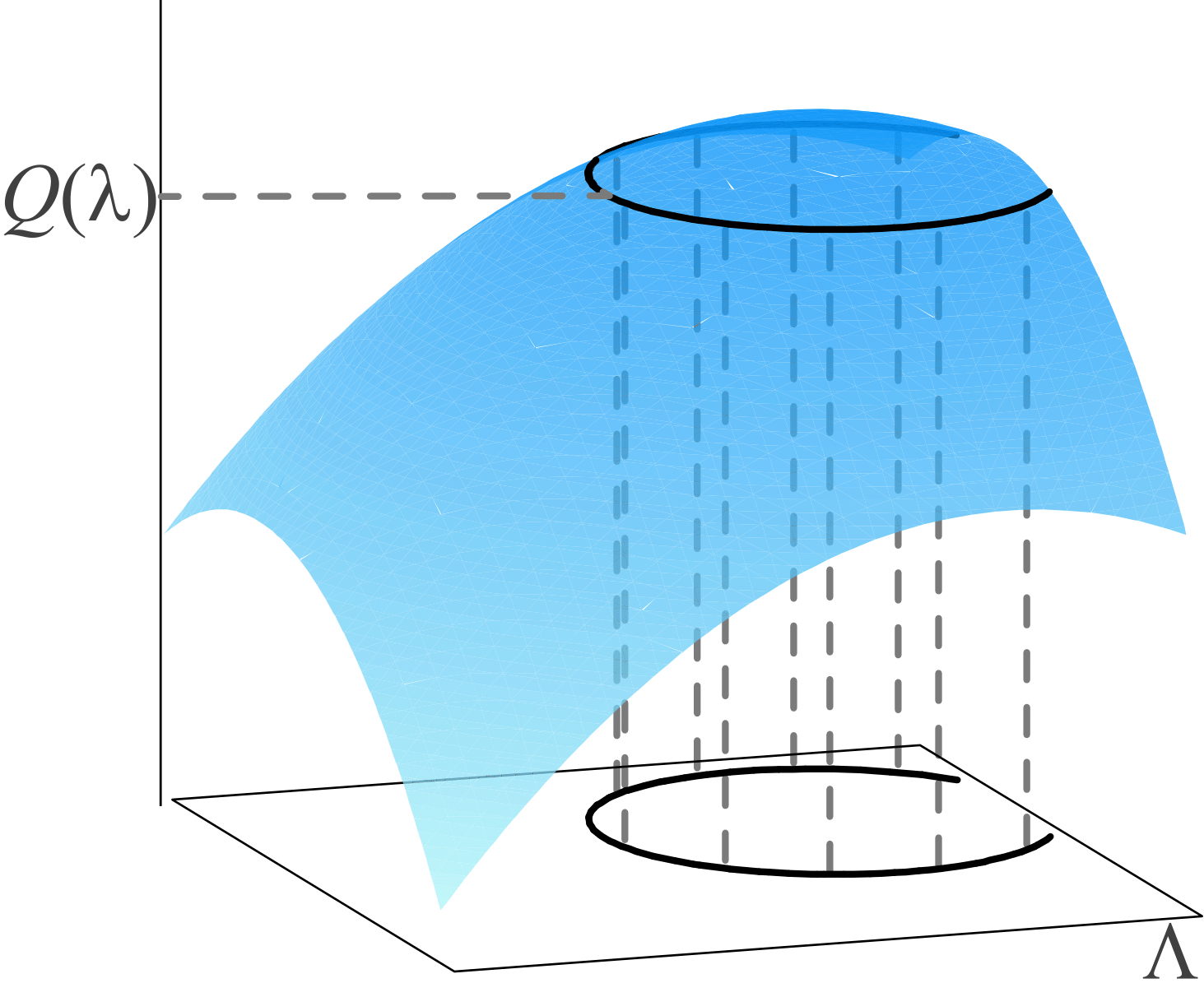}
  \includegraphics[width=0.33\textwidth]{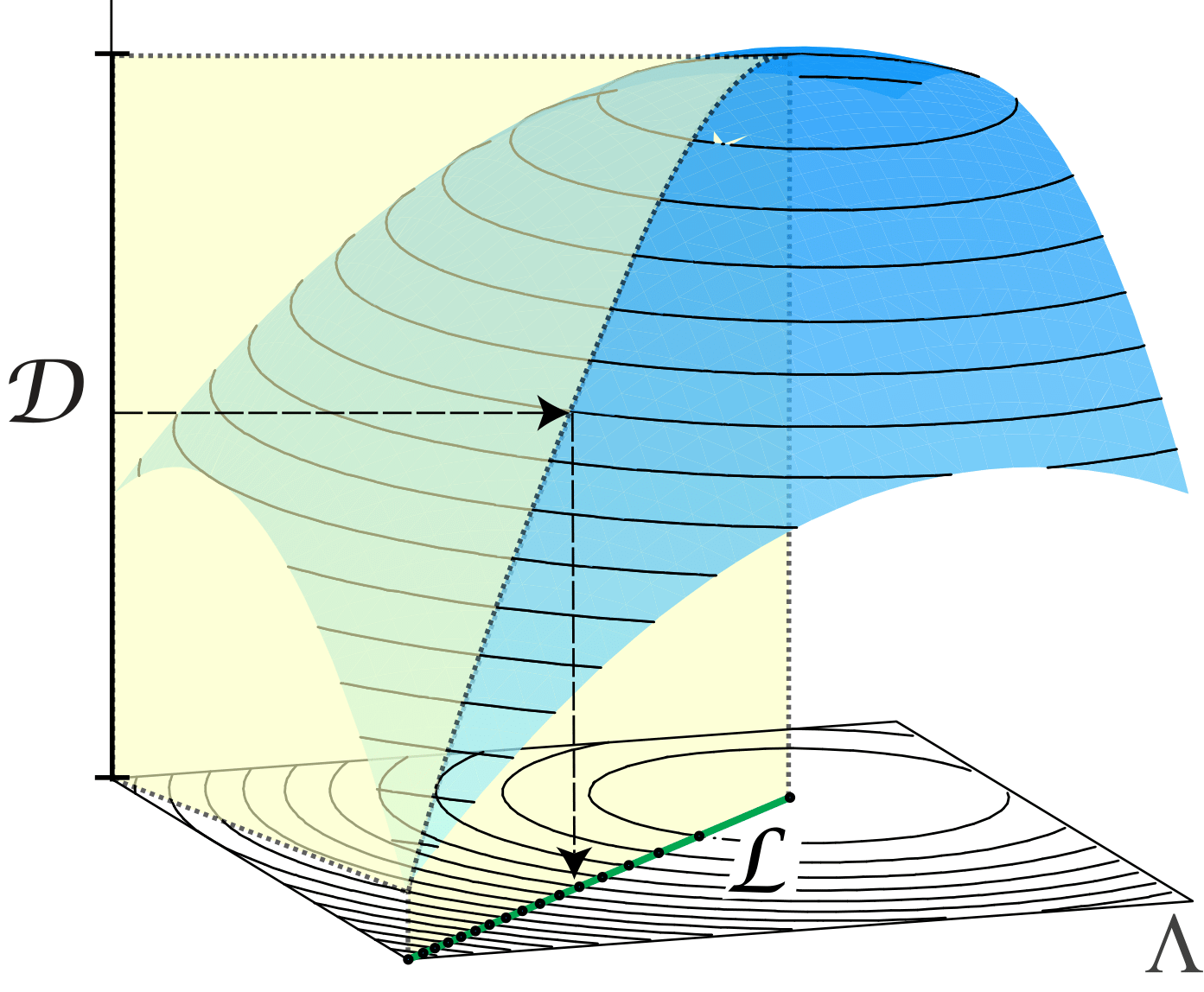}
  \includegraphics[width=0.33\textwidth]{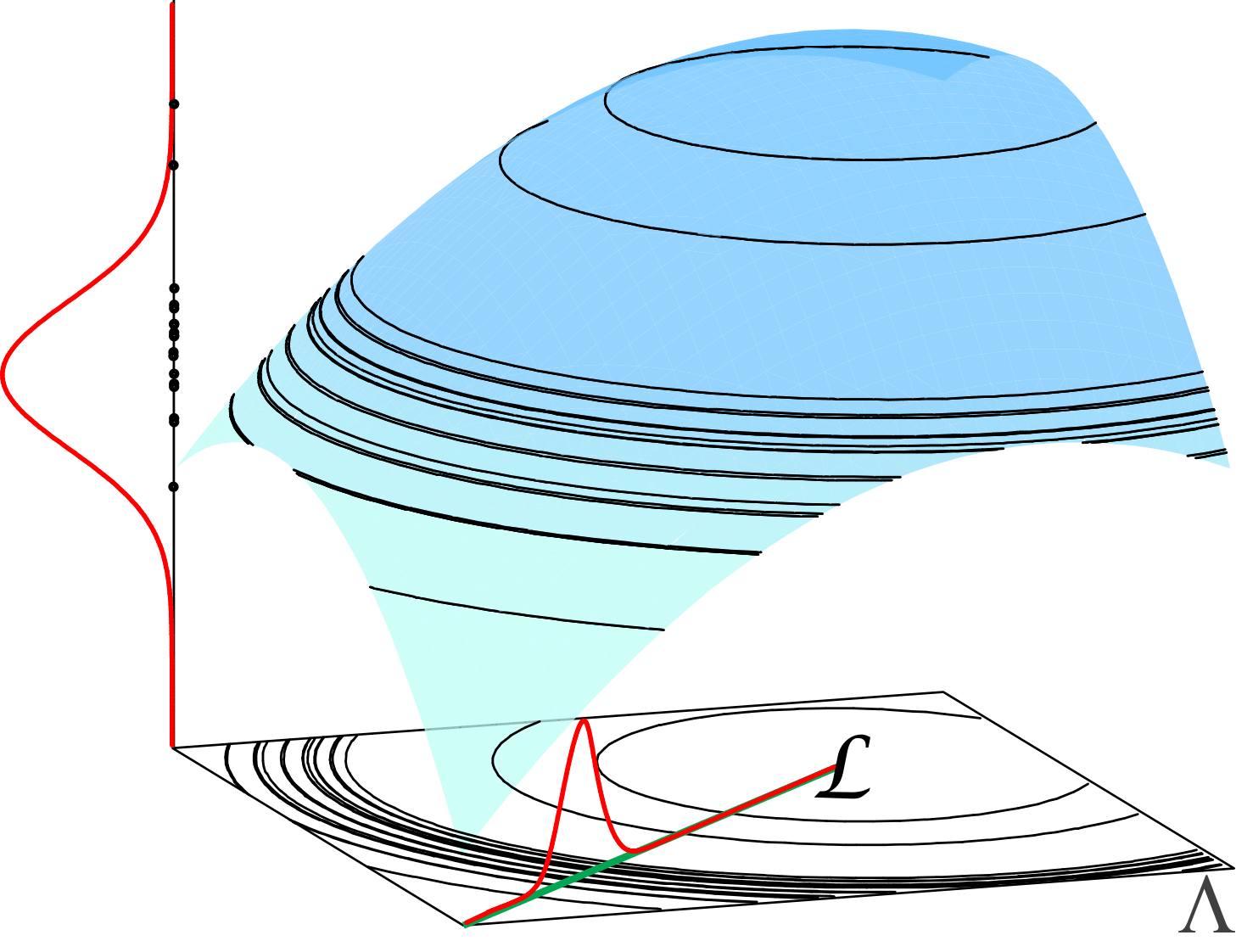}}
  \caption{\it Left: The inverse of $Q^{-1}(Q(\lambda))$ is often set-valued even when $\lambda\in\Lambda$ specifies a particular $Q(\lambda)\in\mathcal{D}$. Middle: The representation of $\mathcal{L}$ as a transverse parameterization. Right: A probability measure described as a density on $(\mathcal{D},\mathcal{B}_{\mathcal{D}})$ maps uniquely to a probability density on $(\mathcal{L},\mathcal{B}_{\mathcal{L}})$. Figures adopted from \cite{BE3} and \cite{BE1}.}\label{fig:contour_illustration}
\end{figure}

\vskip 5pt
{\noindent\it Level 3}

The third forward problem builds on the second forward problem.
We now assume that there is uncertainty, described in terms of probabilities, as to which set $A\in\mathcal{B}_\Lambda$ the model parameters $\lambda$ belong, and the goal is to analyze the probabilities of sets in $\mathcal{B}_{\mathcal{D}}$. 
In the language of probability theory, measurable sets are referred to as events.  
Thus, we assume that a probability measure $P_{\Lambda}$ is given on $(\Lambda,\mathcal{B}_\Lambda)$ describing uncertainty in the events for which parameters may belong, and the goal is to determine $P_{\mathcal{D}}$ on $(\mathcal{D},\mathcal{B}_{\mathcal{D}})$.
We refer to this as the {\em stochastic} forward problem. 
When $P_\Lambda$ (resp., $P_\mathcal{D}$) is absolutely continuous with respect to the volume measure $\mu_{\Lambda}$ (resp., $\mu_{\mathcal{D}}$), the corresponding Radon-Nikodym derivative (i.e., the probability density function) $\rho_{\Lambda}$ (resp., $\rho_{\mathcal{D}}$) is usually given in place of the probability measure.
We assume this is the case so that we can refer to probabilities of events in terms of the more common representation using integrals and density functions, e.g., 
\begin{equation}
    P_{\Lambda}(A) = \int_A \rho_{\Lambda}d\mu_{\Lambda}, \ A\in\mathcal{B}_\Lambda.
\end{equation}
Solution to this stochastic forward problem are given by the induced push-forward probability measure $P_{\mathcal{D}}$ defined for any $A\in\mathcal{B}_{\mathcal{D}}$ by  
\begin{equation}
    P_{\mathcal{D}}(A) = \int_A\rho_{\mathcal{D}}d\mu_{\mathcal{D}} = \int_{Q^{-1}(A)}\rho_{\Lambda}d\mu_{\Lambda} = P_{\Lambda}(Q^{-1}(A)).
\end{equation}
This is a familiar problem in uncertainty quantification.  
The approximate solution can be obtained by a classic Monte Carlo method. 

The corresponding inverse problem assumes we are given a probability measure $P_{\mathcal{D}}$ on $(\mathcal{D},\mathcal{B}_{\mathcal{D}})$, and the goal is to determine a probability measure $P_\Lambda$ on $(\Lambda,\mathcal{B}_\Lambda)$ such that $P_\Lambda(Q^{-1}(A)) = P_{\mathcal{D}}(A)$ for all $A\in\mathcal{B}_{\mathcal{D}}$.  
In other words, the push-forward measure of a solution to the stochastic inverse problem should be $P_{\mathcal{D}}$. 
To determine such a solution to the stochastic inverse problem, we first consider the stochastic inverse problem posed on $(\mathcal{L},\mathcal{B}_{\mathcal{L}})$.

Since $Q$ defines a bijection between $\mathcal{L}$ and $\mathcal{D}$, there is a unique $P_{\mathcal{L}}$ on $(\mathcal{L},\mathcal{B}_{\mathcal{L}})$ (see the right-hand plot in Figure~\ref{fig:contour_illustration} and \cite{BE3}).
We can then use the projection map $\pi_{\mathcal{L}}$ to prove the following theorem.
\begin{theorem}
The stochastic inverse problem has a unique solution on $(\Lambda,\mathcal{C}_{\Lambda})$. 
\end{theorem}

However, the goal is to define a probability measure $P_\Lambda$ on the measurable space $(\Lambda,\mathcal{B}_{\Lambda})$ not on a space involving contour events that are complicated to describe.
This requires an application of the Disintegration Theorem, which allows for the rigorous description of conditional probability measures defined on sets of zero $\mu_\Lambda$-measure \cite{Dellacherie_Meyer, BE3, BET}.
\begin{theorem}\label{thm:ProbabililtyDisintegration}
Let $({\Lambda}, \mathcal{B}_{{\Lambda}})$ be a measurable space. 
Assume that $P_{{\Lambda}}$ is a probability measure on $({\Lambda}, \mathcal{B}_{{\Lambda}})$. 
There exists a family of conditional probability measures $\set{P_{\ell}}$ on $\set{(C_{\ell},\mathcal{B}_{C_\ell})}$ giving the disintegration,
\begin{equation}\label{eq:finaldisintegration}
	P_{{\Lambda}}(A)  = \int_{\pi_{\mathcal{L}}(A)} \bigg( \int_{\pi^{-1}_{\mathcal{L}}(\ell) \cap A} \, dP_{\ell}(\lambda) \bigg) dP_{\mathcal{L}}(\ell), \ \forall A\in\mathcal{B}_{\Lambda}.
\end{equation}
\end{theorem}

Thus, any probability measure on $(\Lambda,\mathcal{B}_{\Lambda})$ can be decomposed into a form involving a probability measure on $(\mathcal{L},\mathcal{B}_{\mathcal{L}})$ uniquely defined by $P_{\mathcal{D}}$ and probability measures on each measurable generalized contour space $(C_{\ell},\mathcal{B}_{C_{\ell}})$ defined by the conditional probabilities $P_{\ell}$.
The conditional probability measures on $\set{(C_{\ell},\mathcal{B}_{C_{\ell}})}$ can not be determined by observations of $Q(\lambda)\in\mathcal{D}$. 
We follow \cite{BE3, BET} and adopt what is referred to as the {\em standard Ansatz} determined by the disintegration of the volume measure $\mu_\Lambda$ to compute probabilities of events inside of a contour event. The standard Ansatz is given by
\begin{equation}\label{eq:StandardAnsatz}
	P_{\ell} = \mu_{C_{\ell}}/\mu_{C_{\ell}}(C_{\ell}), \ \forall\ell\in\mathcal{L},
\end{equation}
where $\mu_{C_\ell}$ is the disintegrated volume measure on the generalized contour $C_\ell$. 

Note that the standard Ansatz can be used as long as $\mu_{\Lambda}(\Lambda)<\infty$.
The approximation method and resulting non-intrusive computational algorithm can be easily modified for any Ansatz.
See \cite{BE3} for more details and theory regarding general choices of the Ansatz.  
Combining an Ansatz with the Disintegration Theorem proves the following theorem.
\begin{theorem}\label{thm:Lambda_inverse}
Under the Ansatz, the stochastic inverse problem has a unique solution on $(\Lambda,\mathcal{B}_{\Lambda})$. 
\end{theorem}

The standard Ansatz results in a probability measure $P_\Lambda$ that inherits key geometric features from the generalized contour map.
The only assumption necessary for employing the standard Ansatz is that a finite volume measure exists, which actually means we only assume that a metric exists as described above and that the diameter of $\Lambda$ is bounded.
Any other choice of Ansatz imposes some other geometric constraints on the solution to the inverse problem that are not coming from the map $Q$.
If we choose not to use any Ansatz, then we can always solve the stochastic inverse problem on $(\Lambda,\mathcal{C}_\Lambda)$, i.e., we can always compute the unique probability measure when restricted to contour events.

\subsection{Numerical Approximation of Solutions}\label{Sec:Numerical_Approximation}

Fundamental to approximating solutions $P_\Lambda$ to the stochastic inverse problem is the approximation of events in the various $\sigma$-algebras, $\mathcal{B}_{\mathcal{D}}$, $\mathcal{C}_\Lambda$ and $\mathcal{B}_\Lambda$. 
Since $\mathcal{C}_\Lambda\subset\mathcal{B}_\Lambda$, we can simultaneously approximate events in both of these $\sigma$-algebras using the same set of events partitioning $\Lambda$.
Let $\set{\mathcal{V}_i}_{i=1}^N$ denote such a partition of $\Lambda$ where $\mathcal{V}_i\in\mathcal{B}_\Lambda$ for each $i$. 
Assume that we are given a collection of sets $\set{D_k}_{k=1}^M\subset\mathcal{B}_{\mathcal{D}}$ partitioning $\mathcal{D}$.
The basic algorithmic procedure for approximating $P_\Lambda(\mathcal{V}_i)$ for each $i$ is to determine which of the $\set{\mathcal{V}_i}_{i=1}^N$s approximate $Q^{-1}(D_k)$, and then to apply the Ansatz on this approximation of the contour event which has known probability $P_{\mathcal{D}}(D_k)$.
Letting $p_{\Lambda,i}$ denote this approximation, we can then define an approximation to $P_{\Lambda}$ as
\[
	P_\Lambda(A) \approx P_{\Lambda,N}(A) = \sum_{i=1}^N p_{\Lambda,i}\chi_{\mathcal{V}_i}(A), \ A\in\mathcal{B}_\Lambda.
\]
We summarize this basic procedure for approximating $P_\Lambda$ in Algorithm~\ref{Alg:1}.

\begin{algorithm}
	\begin{enumerate}[1.]
		\item Let $\{\mathcal{V}_i\}_{i=1}^N\subset\mathcal{B}_\Lambda$ partition $\Lambda$.
		\item Determine a nominal value $Q_i$ for the map $Q(\lambda)$ on $\mathcal{V}_i$ for $i=1,..,N$.
		\item Choose a partitioning of $\mathcal{D}$, $\{ D_k\}_{k=1}^M \subset \mathcal{D}$.
		\item Compute $p_{\mathcal{D},k} =P_{\mathcal{D}}(D_k)$ for $k=1,...,M$.
		\item Let $\mathcal{C}_k = \{i | Q_i \in D_k \}$ for  $k=1,...,M$.
		\item Let $\mathcal{O}_i = \{k | Q_i \in D_k \}$ for $i=1,..,N$.
		\item Compute $V_i= \mu_\Lambda(\mathcal{V}_i)$ for $i=1,..,N$.
		\item Set $p_{\Lambda, i} = (V_{i}/ \sum_{j \in \mathcal{C}_{\mathcal{O}_i}} V_j) p_{\mathcal{D}, \mathcal{O}_i}$ for $i=1,..,N$.
	\end{enumerate}
	\caption{Numerical Approximation of Inverse Probability Measure}\label{Alg:1}
\end{algorithm}

In Algorithm~\ref{Alg:1}, $\mathcal{C}_k$ is used to determine which sets from $\set{\mathcal{V}_i}_{i=1}^N$ approximate the contour event $Q^{-1}(D_k)$ for each $k$.
Similarly, $\mathcal{O}_i$ is used to determine which contour event $Q^{-1}(D_k)$ is associated to $\mathcal{V}_i$ for each $i$.
The Ansatz is applied in the final step where the probability of each $\mathcal{V}_i$ is determined by the probability $P_{\mathcal{D}}(D_{\mathcal{O}_i})$ multiplied by the ratio of the volume of $\mathcal{V}_i$ to the volume of the approximate contour event of $Q^{-1}(D_{\mathcal{O}_i})$.

The ability to accurately approximate contour events by subsets of $\set{\mathcal{V}_i}_{i=1}^N$ is critical in computing accurate approximations of $P_{\Lambda}$.
In other words, the ability to numerically approximate solutions to the deterministic inverse sensitivity problem determines the numerical accuracy in the stochastic inverse problem. 
Refining the partition $\set{\mathcal{V}_i}_{i=1}^N$ generally means additional solutions of the model are required in order to compute the nominal value $Q_i$ required in Step 2 of Algorithm~\ref{Alg:1}.
A typical goal is to determine the minimal number of sets partitioning $\Lambda$ such that $P_\Lambda$ can be approximated for a fixed discretization of $P_{\mathcal{D}}$. 
Such ideas and goals were first explored in \cite{Butler2015b} where the concept of {\em skewness} was defined, which described the computational complexity in solving the inverse problem entirely in terms of the number of sets partitioning $\Lambda$ required for accurate approximation of $P_\Lambda$. 
While there are other approximation issues that arise from numerical evaluation of the map $Q$ and approximation of measures by densities on partitions of the various spaces, these have been addressed elsewhere, e.g., see \cite{BE2, BET}, and are only exacerbated by the fundamental problem of approximating contour events. 
To illustrate the geometric concept of skewness and the affect on the number of samples $N$ needed in Algorithm~\ref{Alg:1} to accurately approximate the inverse image, we use the simple example below, where for simplicity, we assume $n=m$. 

Suppose we are given two different QoI maps $Q^{(a)} = (Q_1^{(a)},Q_2^{(a)}) : \Lambda=[0, 1]^2\goto\mathcal{D}^{(a)}\subset\reals^2$ and $Q^{(b)} = (Q_1^{(b)},Q_2^{(b)}) : \Lambda=[0, 1]^2\goto\mathcal{D}^{(b)}\subset\reals^2$ defined by
\begin{eqnarray}
    Q_1^{(a)}(\lambda_1, \lambda_2) &=& \lambda_1, \\
    Q_2^{(a)}(\lambda_1, \lambda_2) &=& \lambda_2, \\
    Q_1^{(b)}(\lambda_1, \lambda_2) &=& \lambda_1 + \lambda_2, \\
    Q_2^{(b)}(\lambda_1, \lambda_2) &=& 0.74\lambda_1 + 1.26\lambda_2.
\end{eqnarray}
Let $\mathcal{D}^{(a)}$ and $\mathcal{D}^{(b)}$ be partitioned by rectangles $\set{D_k^{(a)}}_{k=1}^{M_a}$ and $\set{D_l^{(b)}}_{l=1}^{M_b}$, respectively.
In Figure~\ref{Fig:voronoi_symmdiff}, we show the affect of skewness by considering $(Q^{(a)})^{-1}(D_k^{(a)})$ and $(Q^{(b)})^{-1}(D_l^{(b)})$ for a fixed $k$ and $l$, the error in the Voronoi approximation of this contour event which is given by the symmetric difference (shown by the shaded regions of green), and how this error decreases with increased $N$.  In Figure~\ref{Fig:loglog_symmdiff} and Table~\ref{Table:symmdiff}, we show the convergence of the measure of the symmetric difference to be near the known convergence rate of MC methods, $N^{-1/2}$.  Although numerical approximations of the inverse densities converge at the same rate for both $Q^{(a)}$ and $Q^{(b)}$, for a given $N$ the error for the well conditioned map ($Q^{(a)}$) is about half the error of the poorly conditioned (or highly skewed) map ($Q^{(b)}$).  We describe how to quantify the skewness of a given map $Q$ in Section~\ref{Sec:Skewness}.

\begin{figure}
    \begin{center}
        \textbf{$Q^{(a)}$ \hskip 155pt $Q^{(b)}$}\par\medskip
        \includegraphics[width=.38\textwidth]{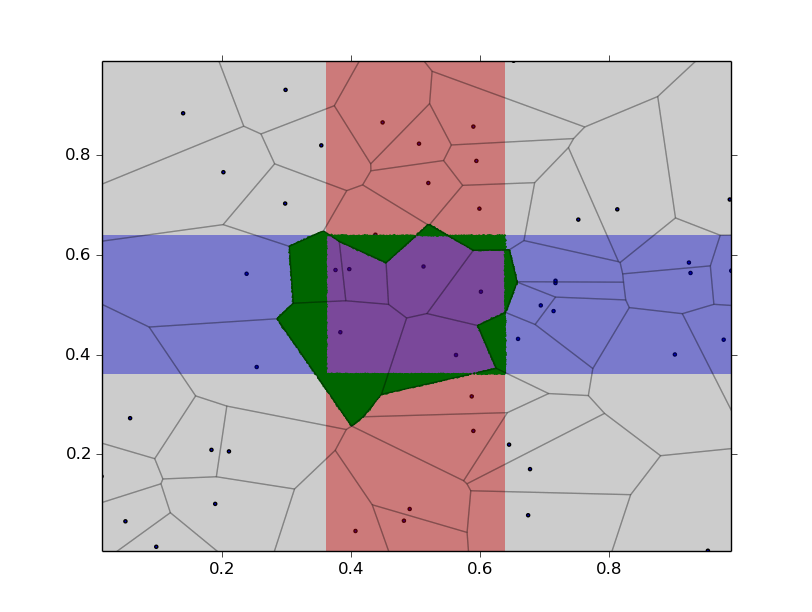}
            \includegraphics[width=.38\textwidth]{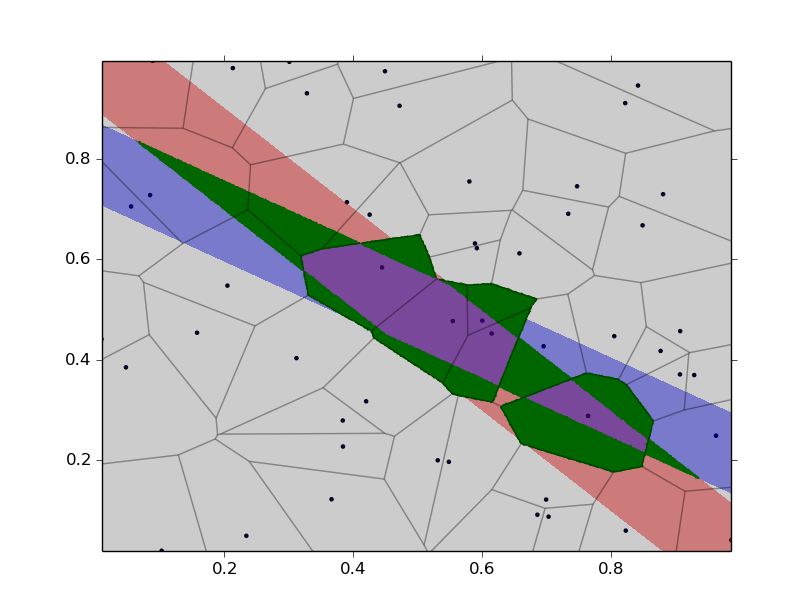}\\
        \includegraphics[width=.38\textwidth]{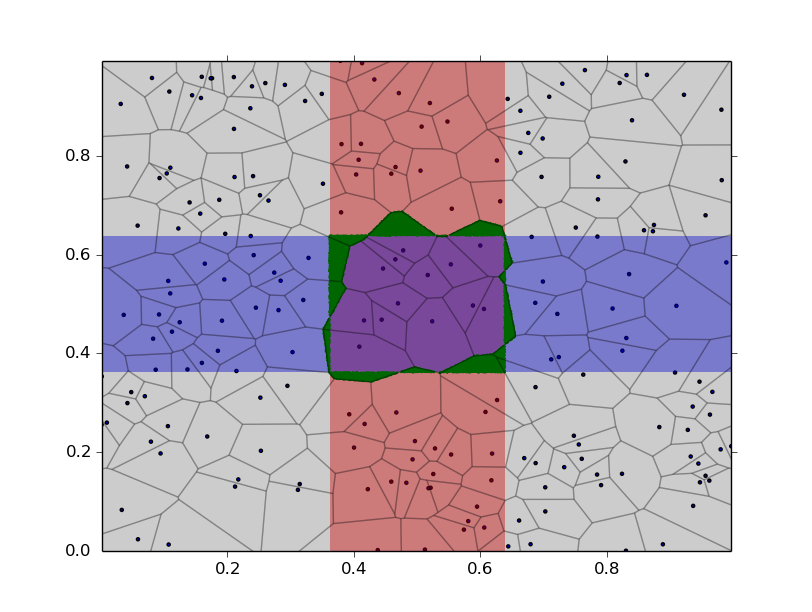}
            \includegraphics[width=.38\textwidth]{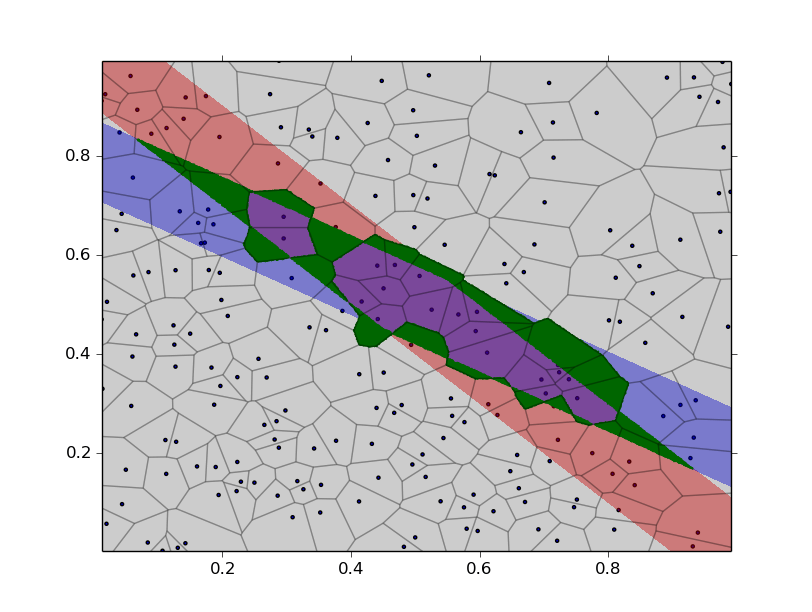}\\
        \includegraphics[width=.38\textwidth]{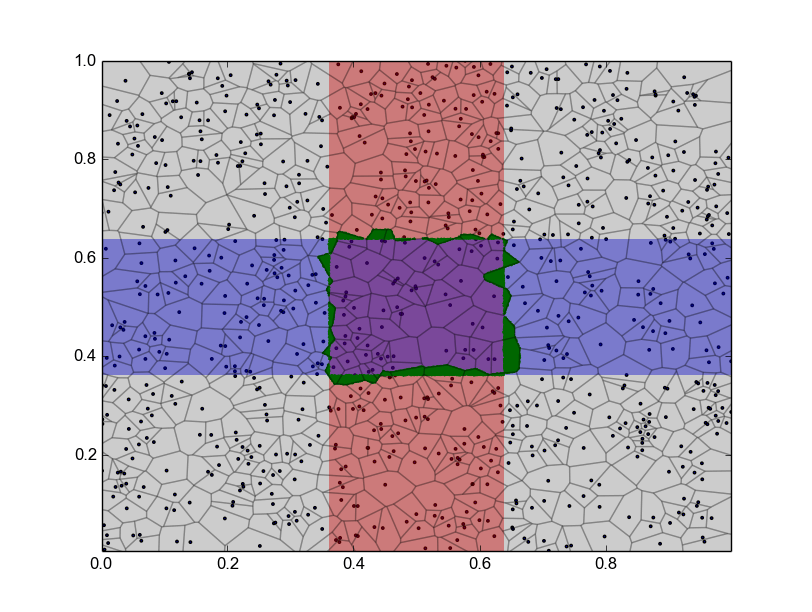}
            \includegraphics[width=.38\textwidth]{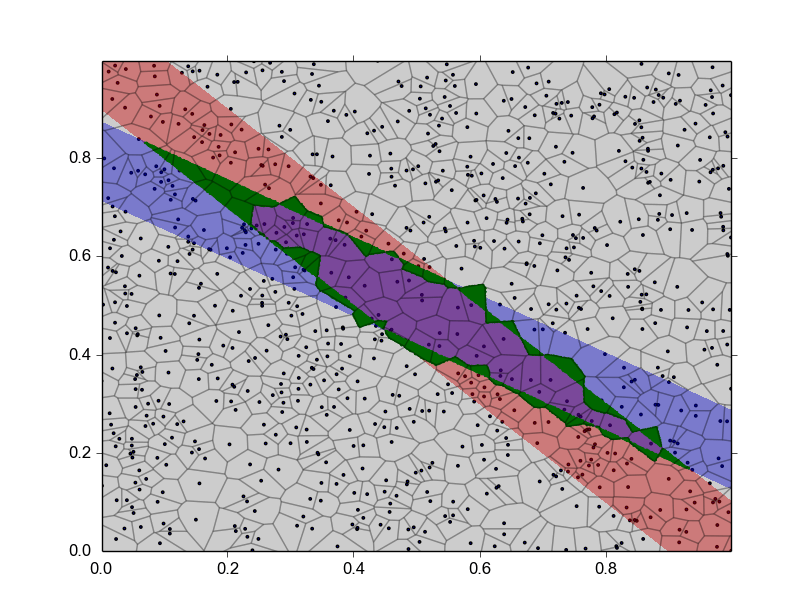}\\
        \includegraphics[width=.38\textwidth]{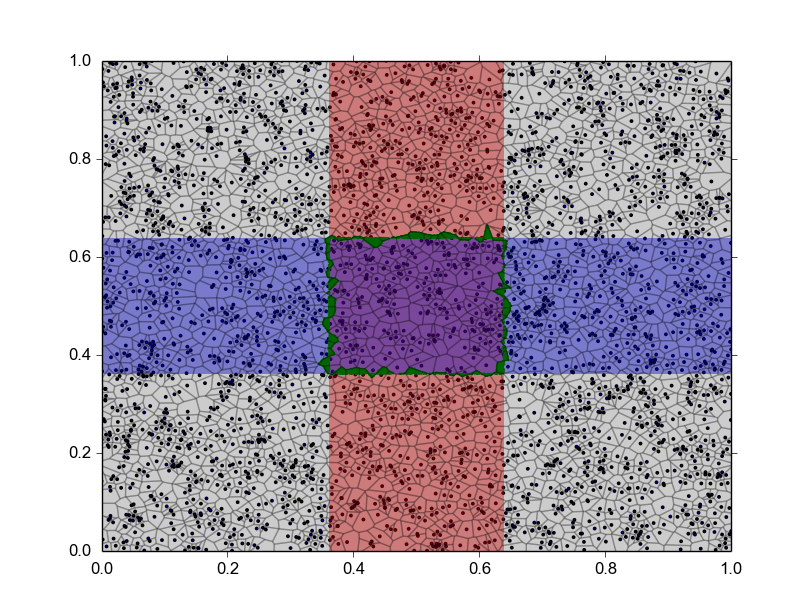}
            \includegraphics[width=.38\textwidth]{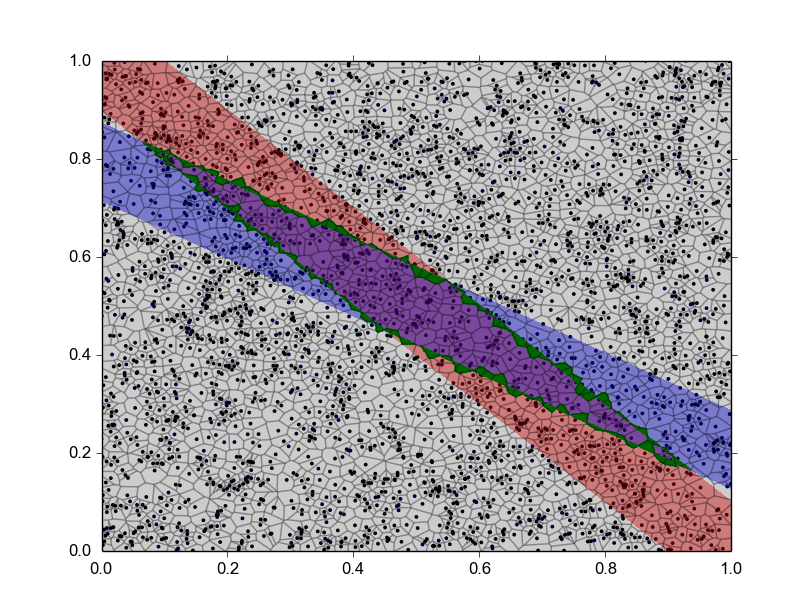}
        \caption{\it We plot the implicit discretization of $\Lambda$ $\{\mathcal{V}_i\}_{i=1}^N$ (given by the uniform random samples), the exact inverse density (given by the intersection of the red and blue contour events) of $D^{(a)}_k$ and $D^{(b)}_l$ for fixed $k$ and $l$, and the symmetric difference of the exact inverese image with the approximated inverse image (seen in green).  We show this for both $Q^{(a)}$ (left) and $Q^{(b)}$ (right).  This is shown for 50 samples (top), 200 samples (second from the top), 800 samples (second from the bottom), and 3200 samples (bottom).}\label{Fig:voronoi_symmdiff}
    \end{center}
\end{figure}

\begin{figure}
	\centering
	\begin{minipage}[t]{.4\textwidth}
	\centering
	\vspace{0pt}
	\includegraphics[width=\textwidth]{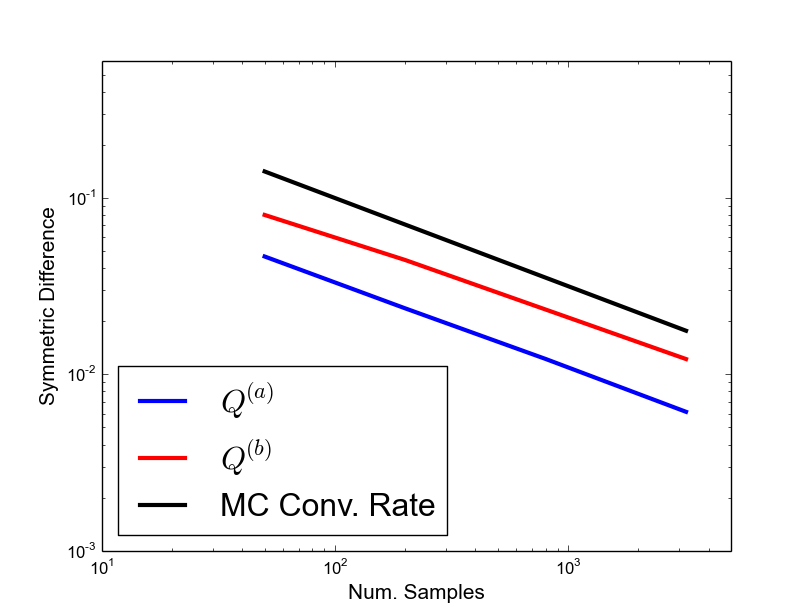}
	\caption{\it Loglog convergence plot for the mean symmetric difference shown in Table~\ref{Table:symmdiff}.}\label{Fig:loglog_symmdiff}
	\end{minipage}
	\hskip 20pt
	\begin{minipage}[t]{.4\textwidth}
	\centering
	\vspace{14pt}
	\begin{tabular}{ccc}
  		Num. Samples & $Q^{(a)}$ & $Q^{(b)}$ \\ \hline \hline
		\vspace{2pt}
  		$50$ &  $4.66E-2$ & $8.01E-2$ \\
		\vspace{2pt}
		$200$ & $2.37E-2$ & $4.45E-2$ \\
		\vspace{2pt}
		$800$ & $1.22E-2$  & $2.33E-2$ \\

		$3200$ & $6.13E-3$  & $1.22E-2$ \\
	\end{tabular}
	\vspace{40pt}
	\captionof{table}{\it Mean of the measure of the symmetric difference for 100 sets of N random samples.}
	\label{Table:symmdiff}
\end{minipage}
\end{figure}

\vskip 10pt
\section{\bf Precision and Accuracy}\label{Sec:Precision_Accuracy}
\vskip 10pt

Following solution to the stochastic inverse problem, we generally want to make, and quantify uncertainty in, predictions that can be formulated as solutions to a stochastic forward problem, e.g., computing a confidence interval on storm surge levels of a hurricane forecast using an ensemble of predictions determined from sampling model input parameters. 
In some cases, we may be able to design the observation network defining the QoI map used in the stochastic inverse problem, e.g., by determining where and when to deploy sensors such as buoys in the Gulf of Mexico.
The problem of defining the ``best'' observation network is generally referred to as optimal experimental design.   
Below, we motivate both the definitions we use to define what is an optimal QoI map as well as the development of criteria for choosing the optimal QoI map. 
Our starting point is the assumption that we solve the stochastic inverse problem in order to {\em quantifiably improve} the predictive capabilities of the model. 

Suppose two different QoI maps, $Q^{(a)}$ and $Q^{(b)}$, define two stochastic inverse problems with solutions represented by the two densities $\rho_\Lambda^{(a)}$ and $\rho_\Lambda^{(b)}$ such that 
\[
	\mu_\Lambda(\supp(\rho_\Lambda^{(a)})) \ll \mu_\Lambda(\supp(\rho_\Lambda^{(b)})).
\] 
For simplicity, assume that $\Lambda\subset\mathbb{R}^n$, $\mu_\Lambda$ is the Lebesgue measure, and the supports of the densities are convex subsets. 
Let $Q^{(p)}$ denote a prediction QoI that is a linear map on $\Lambda$, and let $\mathcal{D}=Q^{(p)}(\Lambda)$ denote the space of possible outcomes for the predictions.
Then, $Q^{(p)}(\supp(\rho_\Lambda^{(a)}))$ and $Q^{(p)}(\supp(\rho_\Lambda^{(b)}))$ define the events of all probable predictions based on solutions to the separate stochastic inverse problems, and $\mu_{\mathcal{D}}(Q^{(p)}(\supp(\rho_\Lambda^{(a)})) \ll \mu_{\mathcal{D}}(Q^{(p)}(\supp(\rho_\Lambda^{(b)}))$.
Consequently, we generally expect that statistical inferences drawn from $Q^{(p)}(\supp(\rho_\Lambda^{(a)}))$ (e.g., the mean of the predictions) to have greater precision in terms of reduced variances, smaller confidence intervals, etc., compared to those drawn from $Q^{(p)}(\supp(\rho_\Lambda^{(b)}))$.

This motivates a general measure-theoretic goal for designing experiments where the goal is to determine an observation network defining the map $Q$ such that solution to the stochastic inverse problem results in a significant amount of probability contained in events of small $\mu_\Lambda$-measure. 
It is clear from Eq.~(\ref{eq:finaldisintegration}), the Ansatz, and Algorithm~\ref{Alg:1}, that the $\mu_\Lambda$-measures of induced contour events play a pivotal role in determining an optimal map $Q$. 
In Section~\ref{Sec:Support}, we describe an efficient computational approach for quantifying the $\mu_\Lambda$-measures of induced contour events for nonlinear maps $Q$.
Given multiple choices for the QoI map, we may use such results to determine the so-called optimal map $Q$.
However, given finite computational resources, it may not be possible to accurately approximate solutions of the stochastic inverse problem using the optimal map $Q$ if only a relatively small number of model solves are allowed as described in Section~\ref{Sec:Numerical_Approximation}.
A computational method for quantifying the global affect of skewness from a nonlinear map $Q$ is described in Section~\ref{Sec:Skewness}. 
In Section~\ref{Sec:Optimizing}, we describe how to optimize the choice of $Q$ to take into account the separate goals of obtaining {\em precise} predictions from stochastic inverse problems that can be solved {\em accurately} with relatively few model evaluations. 

\subsection{$\mu_\Lambda$-measure of support}\label{Sec:Support}

Suppose in an experiment we can obtain a total of $m$ QoI.  Let $B\in\mathcal{B}_{\mathcal{D}}$ define a typical output event from a partition of $\mathcal{D}$ used in Algorithm~\ref{Alg:1}.  If $P_{\mathcal{D}}(B)\approx 1$, then as described above, the goal is to quantify the expected $\mu_\Lambda$-measure of the support of $Q^{-1}(B)$.  Since we do not know $P_{\mathcal{D}}$ a priori, we want to choose a $B$ that is representative of typical tessellations used to discretize any probability measure and also geometrically easy to describe such that $\mu_\Lambda(Q^{-1}(B))$ is computationally inexpensive to approximate.
In measure theory, it is common for $\sigma$-algebras on higher dimensional spaces to be generated from generalized rectangles.
We therefore consider $B$ to be a generalized rectangle in $\mathcal{D}\subset\mathbb{R}^m$. 
Suppose $Q$ is a linear GD map and $n=m$, i.e., $Q$ is defined by the square invertible matrix $J$, $Q(\lambda)=J\lambda$.
In this case, if $\Lambda=\mathbb{R}^n$, we have that
\begin{equation}
	\mu_\Lambda(Q^{-1}(B)) = \mu_\Lambda(J^{-1}(B)) = \mu_{\mathcal{D}}(B)\det(J^{-1}) = \frac{\mu_{\mathcal{D}}(B)}{\det J}.
\end{equation}
Note that if $\Lambda\subset\mathbb{R}^n$ is proper, then the above equation is not necessarily true as $Q^{-1}(B)$ may intersect the boundary of $\Lambda$. 
We neglect such boundary effects in the computations, and simply note that in certain cases they may play an important role although this is not the typical case in our experience. 

In general $Q$ is not linear, so we consider local linear approximations of $Q$ to quantify the $\mu_{\Lambda}$-measure of inverting sets such as $B$ into different regions of $\Lambda$.
Given $\lambda^{(i)}=(\lambda_1^{(i)}, \lambda_2^{(i)} \ldots, \lambda_n^{(i)})\in\Lambda$, denote the Jacobian of $Q$ evaluated at $\lambda^{(i)}$, by
\begin{eqnarray}
    J_{\lambda^{(i)}} = \bcm \frac{\partial Q_1(\lambda^{(i)})}{\partial \lambda_1} & \hdots & \frac{\partial Q_1(\lambda^{(i)})}{\partial \lambda_n} \\
        \vdots & \ddots & \vdots \\
        \frac{\partial Q_m(\lambda^{(i)})}{\partial \lambda_1} & \hdots & \frac{\partial Q_m(\lambda^{(i)})}{\partial \lambda_n} \ecm.
\end{eqnarray}
For simplicity in describing the sets $Q^{-1}(B)$ we initially assume $n=m$ and the map $Q$ is GD, which implies that the Jacobian is square and invertible.  Suppose now that the generalized rectangle $B$ is centered at $Q(\lambda^{(i)})$, then from above we have that
\begin{equation}
    \mu_\Lambda(Q^{-1}(B)) \approx M_Q(\lambda^{(i)}):=\mu_\Lambda(J_{\lambda^{(i)}}^{-1}(B))  = \frac{\mu_{\mathcal{D}}(B)}{\det J_{\lambda^{(i)}}}.
\end{equation}
Here, we introduce the notation $M_Q(\lambda^{(i)})$ as a shorthand for the the size (i.e., measure) of local induced contour events defined around a sample $\lambda^{(i)}$, and we make explicit the dependence on the QoI map $Q$. 
Since the determinant of $J$ can be written as the product of the singular values of $J$, we have
\begin{equation}\label{Eq:prod_singvals}
    M_Q(\lambda^{(i)}) = \mu_{\mathcal{D}}(B)\prod_{k=1}^{m} \, \sigma^*_{ik} = \frac{\mu_{\mathcal{D}}(B)}{\prod_{k=1}^{m} \, \sigma_{ik}}, 
\end{equation}
where $\set{\sigma^*_{ik}}_{k=1}^m$ are the singular values of $J_{\lambda^{(i)}}^{-1}$ and $\set{\sigma_{ik}}_{k=1}^m$ are the singular values of $J_{\lambda^{(i)}}$.   
The {\it average $M_Q(\lambda)$} is given by
\begin{equation}\label{Eq:suppint}
	\overline{M_Q} = \frac{1}{\mu_{\Lambda}(\Lambda)}\int_{\Lambda} M_Q(\lambda) \, d\mu_{\Lambda}.
\end{equation}
Given a set of $N$ samples $\set{\lambda^{(i)}}_{i=1}^N$ in $\Lambda$, we approximate $\overline{M_Q}$ using the Monte Carlo estimate
\begin{equation}\label{Eq:supp}
	\overline{M_Q} \approx \overline{M_{Q,N}} := \frac{1}{N}\sum_{i=1}^N M_Q(\lambda^{(i)}) = \frac{1}{N}\sum_{i=1}^N \frac{\mu_{\mathcal{D}}(B)}{\prod_{k=1}^m\sigma_{ik}}.
\end{equation}

We now relax the assumption that $n=m$ and consider the (more common) case of $m<n$.  
In this case, we recall that a transverse parameterization, which is an explicit representation in $\Lambda$ of the equivalence class structure $\mathcal{L}$ defined by the generalized contours, exists as a (possibly piecewise defined) $m$-dimensional manifold.  
Restricting $Q:\Lambda\to\mathcal{D}$ to be $Q:\mathcal{L}\to\mathcal{D}$, which we denote by $Q|_{\mathcal{L}}$, we then have an $m\times m$ map and can apply the above computations to $Q|_{\mathcal{L}}$. 
In other words, we are now interested in the $\mu_{\mathcal{L}}$-measure of $Q|_{\mathcal{L}}^{-1}(B)\in\mathcal{B}_{\mathcal{L}}$.  
Geometrically, we may think of this problem as determining the measure of the {\em cross section of the contour event intersected with a transverse parameterization manifold} in $\Lambda$. 
Given a local linear approximation of $Q$ at some point $\lambda^{(i)}\in\Lambda$, the following lemma shows that we may approximate $\mu_{\mathcal{L}}(M_{Q|_{\mathcal{L}}}(\lambda^{(i)}))$ similar to before. 
Consequently, the lemma implies that we may compute $\overline{M_{Q|_{\mathcal{L}},N}}$ using $M_{Q|_{\mathcal{L}}}(\lambda^{(i)})$ in Eq.~(\ref{Eq:supp}).

\begin{lemma}\label{Lemma:prod_singvals}
Let $J$ be a full rank $m\times n$ matrix with $m\leq n$.  Then the Lebesgue measure $\mu$ in $\reals^m$ of the $m$-dimensional parallelepiped defined by the $m$ rows of $J$, denoted $\mu(Pa(J))$, is given by the product of the singular values of $J$,
\begin{equation}
    \mu(Pa(J)) = \prod_{i=1}^m\sigma_i.
\end{equation}
\end{lemma}

\begin{proof}
The singular values of $J$ are equal to the singular values of $J^{\top}$.  Consider the reduced QR factorization of $J^{\top}$,
\begin{equation}
    J^{\top} = \tilde QR,
\end{equation}
where $\tilde Q$ is $n\times m$ and $R$ is $m\times m$.  
By the properties of the QR factorization, we know the singular values of $R$ are the same as the singular values of $J^{\top}$.  Let $x\in\reals^m$, then
\begin{equation}
	||\tilde Qx||^2=(\tilde Qx)^{\top}(\tilde Qx)=x^{\top}\tilde Q^{\top}\tilde Qx=x^{\top}x=||x||^2,
\end{equation}
so $\tilde Q$ is an isometry.  This implies the Lebesgue measure of the parallelepiped defined by the rows of $R$ is equal to the Lebesgue measure of the parallelepiped defined by the columns of $J^{\top}$, or the rows of $J$,
\begin{equation}
    \mu(Pa(J)) = \mu(Pa(R)) = \prod_{k=1}^m\gamma_k = \prod_{k=1}^m\sigma_k,
\end{equation}
where $\gamma_k$ are the singular values of $R$ and $\sigma_k$ are the singular values of $J$.
\end{proof}

\subsection{Skewness}\label{Sec:Skewness}

While determining the QoI map $Q$ that optimally reduces $\mu_\Lambda(Q^{-1}(B))$ for $B\in\mathcal{B}_\mathcal{D}$ should reduce uncertainty (i.e., increase precision) in predictions, we often need to computationally approximate $Q^{-1}(B)$ as described in Algorithm~\ref{Alg:1}. Errors in the approximations to $P_\Lambda$ propagate to errors in predictions, and while the predicted sets of high probability may have small $\mu_{\mathcal{D}}$-measure, the errors in the prediction may be so large as to render such predictions useless in practice.  Hence, in this section we quantify the {\em skewness} of local induced contour maps and then expand this to the {\em average skewness} as done in Section~\ref{Sec:Support} for the $\mu_\Lambda$-measure, $M_Q$.

Again, we refer to local linear approximations of the given map $Q$ at some $\lambda^{(i)}\in\Lambda$, $J_{\lambda^{(i)}}$.  Let $Pa(J_{\lambda^{(i)}})$ denote the parallelepiped defined by the $n$-dimensional vectors $j_1, \hdots, j_m$ that are the rows of $J_{\lambda^{(i)}}$.  We use $\mu(Pa(J_{\lambda^{(i)}}))$ to denote the measure of the parallelepiped as an $m$-dimensional object.  This measure is calculated in terms of the singular values of $J_{\lambda^{(i)}}$ by Lemma \ref{Lemma:prod_singvals}.  A fundamental decomposition is

\begin{theorem}
    Given $n$-dimensional vectors $j_1, \hdots, j_m$, there exists vectors $j_1^{\perp}, j_1^0$ such that
    \begin{equation*}
        j_1 = j_1^{\perp} + j_1^0, \hskip 10pt j_1^{\perp}\perp j_1^0, \hskip 10pt j_1^0\in span\{j_2, \hdots, j_m\},    
    \end{equation*}
    and
	\begin{equation}\label{Eq:mu_ji}
        \mu(Pa(J_{\lambda^{(i)}})) = |j_1^{\perp}| \times \mu(Pa(J_{1,\lambda^{(i)}})),
	\end{equation}
	where $J_{\lambda^{(i)}}$ denotes the $m\times n$ matrix formed by the vectors $j_1, \hdots, j_m$, $J_{1,\lambda^{(i)}}$ denotes the matrix $J_{\lambda^{(i)}}$ with the first row deleted, and $\mu(Pa(J_{1,\lambda^{(i)}}))$ denotes the measure of the parallelepiped defined by the $n$-dimensional vectors $j_2, \hdots, j_m$ as an $(m-1)$-dimensional object.
\end{theorem}
With this fundamental decomposition we define the local skewness of a given map $Q$ at some point $\lambda^{(i)}\in\Lambda$, see \cite{Butler2015b}.

\begin{definition}
    For a vector $j_k$, we define
	\begin{equation}\label{Eq:skew}
        S_Q(J_{\lambda^{(i)}}, j_k) = \frac{|j_k|}{|j_k^{\perp}|}.
	\end{equation}
    Then we define the {\bf local skewness} at a point $\lambda^{(i)}$ as
	\begin{equation}\label{Eq:Skew}
        S_Q(J_{\lambda^{(i)}}) = \max_k \, S_Q(J_{\lambda^{(i)}}, j_k).
	\end{equation}
\end{definition}

\begin{corollary}\label{Cor:singular_values}
    The local skewness of $Q$, $S_Q(J_{\lambda^{(i)}})$, can be completely determined by the the norms of $n$-dimensional vectors and products of singular values of the Jacobians of QoI maps of dimenions $m-1$ and $m$,
    	\begin{equation}
    		S_Q(J_{\lambda^{(i)}}) = \max_k \frac{|j_k|\mu(Pa(J_{k,\lambda^{(i)}}))}{\mu(Pa(J_{\lambda^{(i)}}))}.
    \end{equation}
\end{corollary}

\begin{proof}
	\begin{equation}
		S_Q(J_{\lambda^{(i)}}) = \max_k \, S_Q(J_{\lambda^{(i)}}, j_k) = \max_k \, \frac{|j_k|}{|j_k^{\perp}|} = \max_k \, \frac{|j_k|\mu(Pa(J_{k,\lambda^{(i)}}))}{\mu(Pa(J_{\lambda^{(i)}}))},
	\end{equation}
	then applying Lemma~\ref{Lemma:prod_singvals} we have
	\begin{equation}
		\max_k \frac{|j_k|\mu(Pa(J_{k,\lambda^{(i)}}))}{\mu(Pa(J_{\lambda^{(i)}}))} = \max_k \, \frac{|j_k|\prod_{r=1}^{m-1}\sigma_{kr}}{\prod_{r=1}^{m}\sigma_r}.
	\end{equation}
	where $\sigma_r$ are the singular values of $J_{\lambda^{(i)}}$ and $\sigma_{kr}$ are the singular values of $J_{k,\lambda^{(i)}}$.
\end{proof}

Corollary \ref{Cor:singular_values} allows us to take advantage of efficient singular value decompositions to algorithmically approximate $S_Q(J_{\lambda^{(i)}})$.  The average skewness is given by
\begin{equation}\label{Eq:skewint}
	\overline{S_Q} = \frac{1}{\mu_{\Lambda}(\Lambda)}\int_{\Lambda} S_Q(\lambda) \, d\mu_{\Lambda},
\end{equation}
and approximated by
\begin{equation}
    \overline{S_Q}\approx\overline{S_{Q,N}} := \frac{1}{N}\sum_{i=1}^{N}S_Q(\lambda^{(i)}).
\end{equation}
\begin{remark}
	Notice the definition of average local skewness is independent of $B\in\mathcal{B_D}$.  This is because $B$ is assumed to be a generalized rectangle, i.e., has perfect skewness properties itself.  This is a reasonable assumption as $B$ is typically defined by the cross product of intervals defining the uncertainty in each QoI.
\end{remark}

\section{\bf Multicriteria Optimization}\label{Sec:Optimizing}

A common problem in designing observation networks is to determine where and when to deploy a finite number of sensors in order to record {\em useful} data related to a physics-based model.
The various possible configurations of sensors defines a family of possible QoI maps. 
Given a family of possible QoI maps, our goal is to determine a particular QoI map resulting in {\em precise} and {\em accurate} numerical approximation of the inverse solution.
Here, we frame the problem of determining such a QoI map  as an optimization problem that simultaneously reduces both $\overline{M_Q}$ and $\overline{S_Q}$.

Let $\mathcal{Q}$ denote the space of all possible sets of QoI, let $Q^{(z)}\in\mathcal{Q}$ represent a set of $m$ QoI, and let $\mathcal{D}^{(z)}$ represent the $m$-dimensional data space defined by $Q^{(z)}$, i.e., $Q^{(z)} : \Lambda\subset\reals^n\goto\mathcal{D}^{(z)}\subset\reals^m$.  
For each $Q^{(z)}$ there is a corresponding generalized rectangle $B^{(z)}\in\mathcal{B}_{\mathcal{D}^{(z)}}$ that represents the uncertainty in a possible recorded datum in $\mathcal{D}^{(z)}$.  

Recall from Section \ref{Sec:Precision_Accuracy} the goals are to reduce $\overline{S_{Q^{(z)}}}$ and $\overline{M_{Q^{(z)}}}$. 
Let $\mathcal{S}\subset\mathbb{R}$ denote the set of all possible values of $\overline{S_{Q^{(z)}}}$ and $\mathcal{M}\subset\mathbb{R}$ denote the set of all possible values of $\overline{M_{Q^{(z)}}}$.  
We then define the metric spaces $(\mathcal{S}, d_{\mathcal{S}})$ and $(\mathcal{M}, d_{\mathcal{M}})$ where
\begin{equation}
	d_{\mathcal{S}}(x, y) = \frac{|x-y|}{1 + |x-y|} \, \text{for all } x,y\in\mathcal{S},
\end{equation}
and
\begin{equation}
	d_{\mathcal{M}}(x, y) = \frac{|x-y|}{1 + |x-y|} \, \text{for all } x,y\in\mathcal{M}.
\end{equation}
We define the metrics this way so that the effect of the scaling of each component is limited in solving the minimization problem defined below. 
Let $Y_\omega$ denote the 2-dimensional product space $\mathcal{S} \times \mathcal{M}$, with metric defined by
\begin{equation}\label{Eq:dist_suppskew}
	d_{Y_\omega}(x, y) = \omega d_{\mathcal{S}}(x_1, y_1) + (1-\omega )d_{\mathcal{M}}(x_2, y_2) \, \text{for all } x,y\in Y_\omega,
\end{equation}
where $\omega\in[0,1]$ determines the relative importance we place on either precision or accuracy.  

\begin{remark}
	The weighting of the $\mu_\Lambda$-measure and the skewness is a current topic of discussion.  When more computational resources are available for solving the model, it is likely we weight the skewness less than the $\mu_\Lambda$-measure.  When we can only afford few samples we weight the skewness more than $\mu_\Lambda$-measure.  This is a topic for future work.
\end{remark}

With this notion of distance on $Y_\omega$ we define the multicriteria optimization problem as finding the solution to,
\begin{equation}\label{eq:optimal_Q}
	\min_{Q^{(z)}\in\mathcal{Q}} d_{Y_\omega}(p, y_z),
\end{equation}
where $p=(1, 0)$ is the {\it ideal point} and $y_z=(\overline{S_{Q^{(z)}}}, \overline{M_{Q^{(z)}}})$.

\subsection{Discrete Optimization}\label{Sec:Discrete}

Suppose $\mathcal{Q}$ is a finite family of possible QoI maps, e.g., as defined by identifying a finite set of $d$ physically possible configurations of $m$ sensors with $d\geq m$. 
The solution to the minimization problem defined by Eq.~\eqref{eq:optimal_Q} can be found by an exhaustive search through the $d \choose m$ possible QoI maps.
This combinatorial problem can clearly become computationally expensive as the number of possible maps gets large. 
However, it is completely straightforward to implement and is embarrassingly parallel. 
In the concluding remarks of Section~\ref{Sec:Conclusion}, we describe some possible future directions to mitigate the cost of solving the optimization problem. 
We provide two numerical examples to clarify this method of determining an optimal QoI map to use in the inverse problem.

\begin{remark}
We note that another equivalent way to frame the discrete optimization problem is to define a ``theoretical'' QoI map $Q:\Lambda\to\mathcal{D}\subset\mathbb{R}^d$ where the goal is to determine the ``practical'' QoI map given by a subset of $m$ components of the map $Q$.
The advantage of this approach from an implementation point of view is that for each sample in $\Lambda$, the model is solved once in order to compute all the possible QoI values.
It also provides an index to each possible QoI making the description of the optimal QoI more straightforward.
Therefore, we use this in the numerical examples below. 
\end{remark}

\subsection{Numerical Examples}\label{Sec:1}

\subsubsection{Linear Example}\label{Sec:Linear}

Let $Q : \reals^2 \goto \reals^3$,
\begin{eqnarray}
    Q = \bcm 0.5 & 0.5 \\
        		 2.5 & 0.5 \\
        		 -0.2 & 0.3  \ecm.
\end{eqnarray}
Notice the rows of $Q$ are pairwise linearly independent, i.e., pairwise GD.  Let $B^{(z)}\in \mathcal{B}_{\mathcal{D}^{(z)}}$ be the rectangle of uncertainty in the data space defined by the $z^{th}$ possible pair of QoI.  For simplicity, we let the uncertainty in each QoI be the same, i.e., $\mu_{\mathcal{D}^{(z)}}(B^{(z)})$ is constant for all $z$.  
Note that because $Q$ is linear we can easily use the exact Jacobian of $Q$.  
Furthermore, we use the exact averages $\overline{M_{Q^{(z)}}}$ and $\overline{S_{Q^{(z)}}}$ not the approximate averages. 
The linearity of $Q$ implies $M_{Q^{(z)}}(\lambda^{(i)})$ and $S_{Q^{(z)}}(\lambda^{(i)})$ are each constant for all $i$, so we need not numerically approximate the integrals in Eqs.~\eqref{Eq:suppint} and \eqref{Eq:skewint}.  
In Section~\ref{Sec:Nonlinear} we consider a nonlinear map and compute approximations of the Jacobian, $\overline{M_{Q^{(z)}}}$, and $\overline{S_{Q^{(z)}}}$.

We see in Figure \ref{Fig:linear_2to3} and Table~\ref{Table:linear_2to3} the optimal choice of pair of QoI to use in the inverse problem for three different optimization problems; minimize $\overline{{M_{Q^{(z)}}}}$, minimize $\overline{{S_{Q^{(z)}}}}$, and minimize Eq.~\eqref{Eq:dist_suppskew} with $\omega = 0.5$.  In the top row of Figure~\ref{Fig:linear_2to3} we see the inverse image of $B^{(z)}$ as the intersection of the red and blue contour events corresponding to the individual components of the possible QoI maps.  
It is visually evident from the top-left plot of Figure~\ref{Fig:linear_2to3} that the pair of QoI that minimizes $\overline{M_{Q^{(z)}}}$ also produces the inverse image with the sharpest corners, i.e., it has the highest skewness.  
In the top-middle plot of Figure~\ref{Fig:linear_2to3} we observe the opposite effect, the pair of QoI that minimizes the skewness $\overline{S_{Q^{(z)}}}$ maximizes $\overline{M_{Q^{(z)}}}$. 
In the top-right plot of Figure~\ref{Fig:linear_2to3}, we observe a pair of QoI that produces reasonably low $\overline{S_{Q^{(z)}}}$ and $\overline{M_{Q^{(z)}}}$ simultaneously.

\begin{figure}
	\centering
	\begin{minipage}[t]{\textwidth}
	\centering
	\vspace{0pt}
        \includegraphics[width=.3\textwidth]{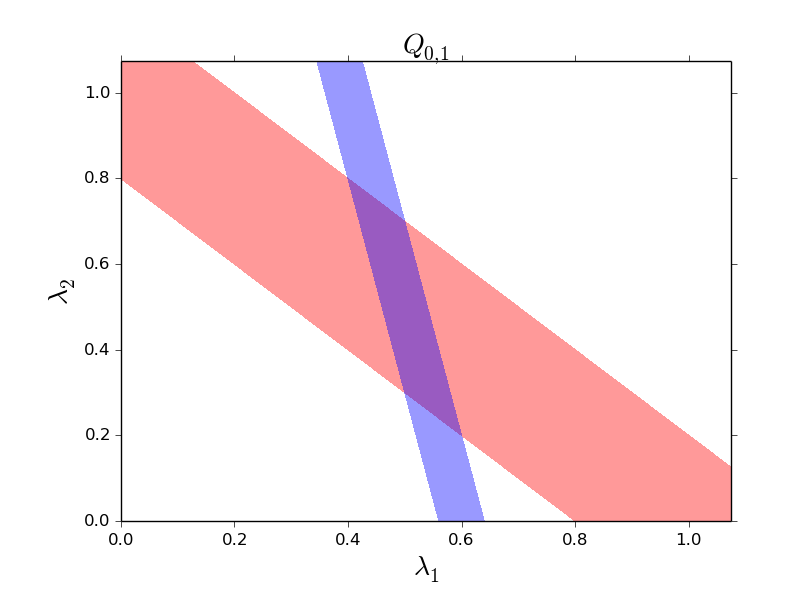}
			\includegraphics[width=.3\textwidth]{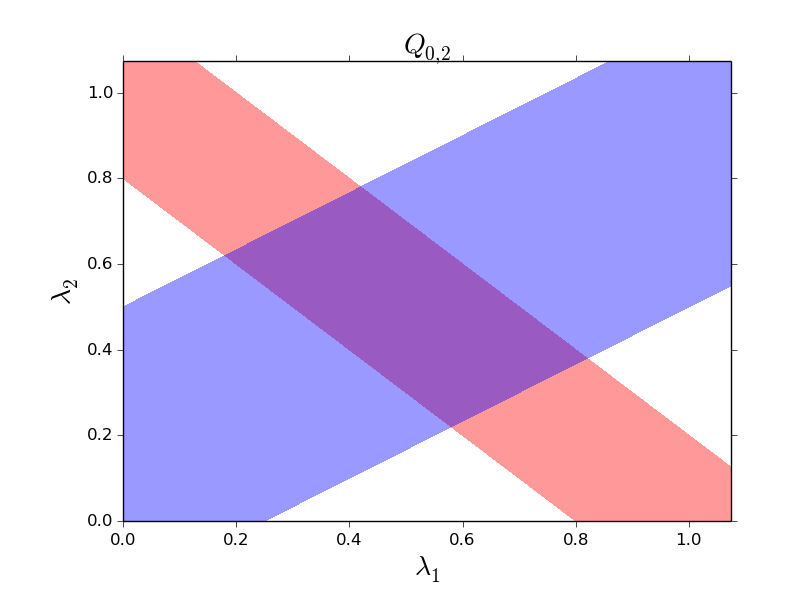}
				\includegraphics[width=.3\textwidth]{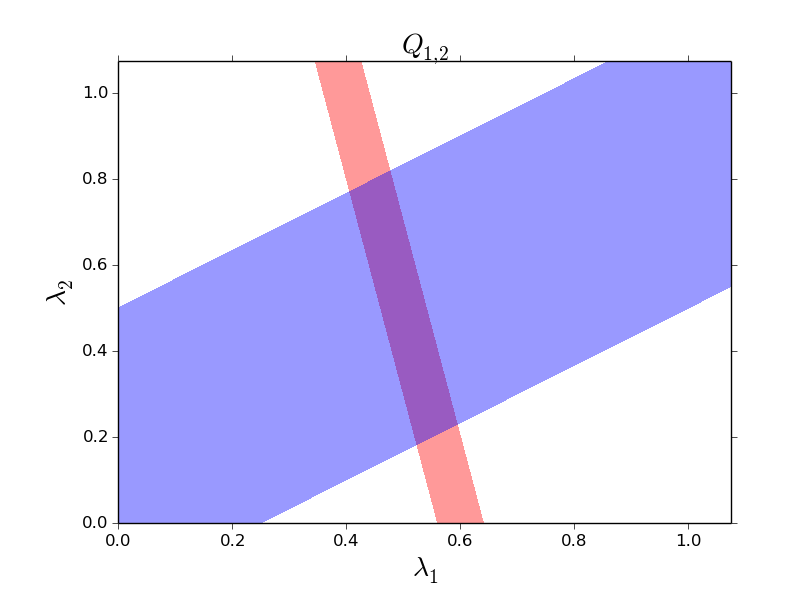}\\
			        
        	\includegraphics[width=.3\textwidth]{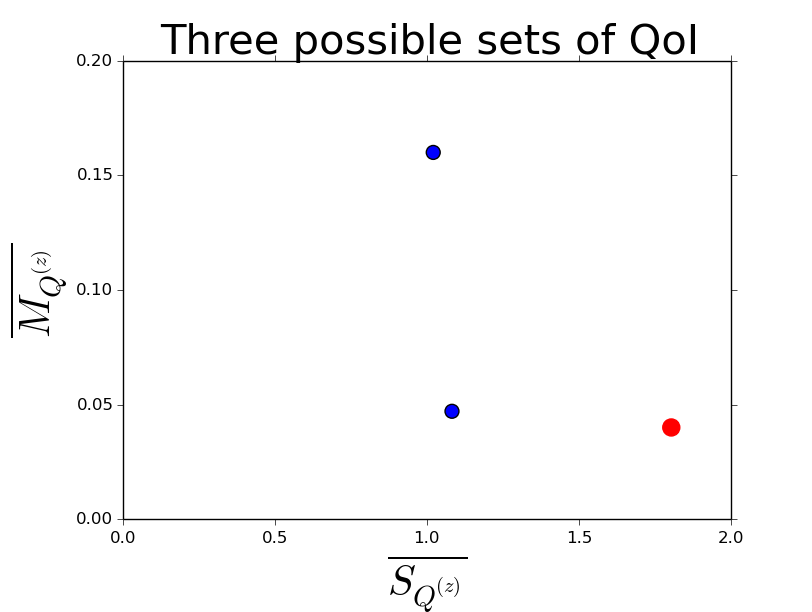}
			\includegraphics[width=.3\textwidth]{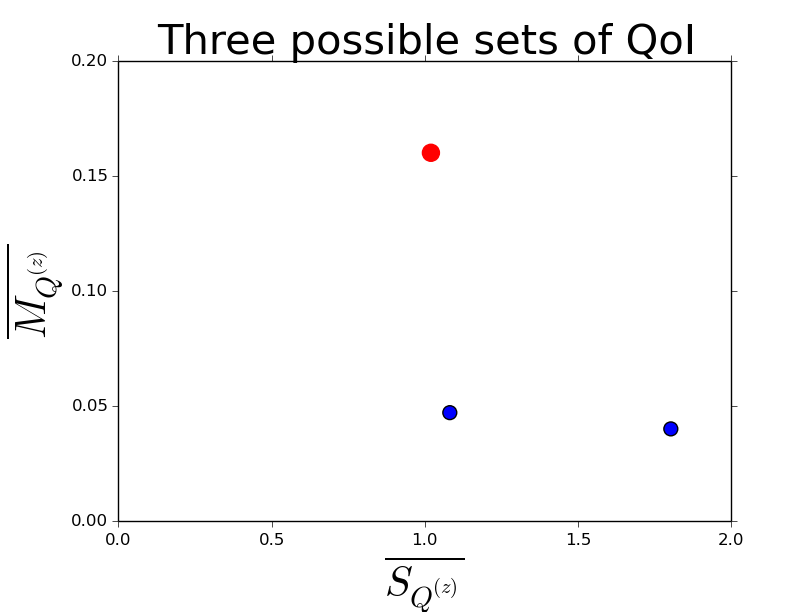}
        			\includegraphics[width=.3\textwidth]{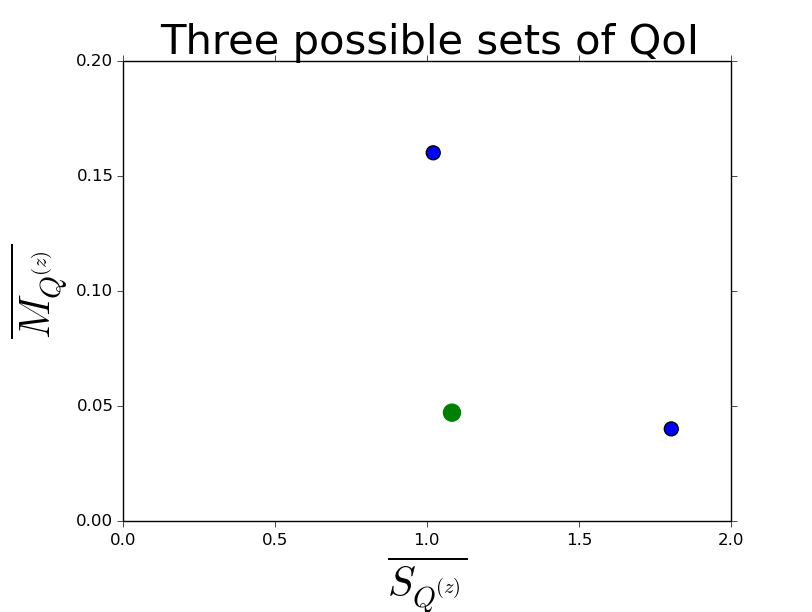}
	\caption{\it On the top we show the inverse density for the three possible pairs of QoI.  On the bottom we show the point in $Y_\omega$ defined by each pair of QoI.  (left): The pair of QoI that minimizes $\overline{M_{Q^{(z)}}}$ ($Q_{0,1}$).  (middle): The pair of QoI that minimizes $\overline{S_{Q^{(z)}}}$ ($Q_{0,2}$).  (right): The pair of QoI that minimizes the distance defined in Eq.(\ref{Eq:dist_suppskew}) with $\omega=0.5$ ($Q_{1,2}$). }\label{Fig:linear_2to3}
	\end{minipage}
	\begin{minipage}[t]{\textwidth}
	\centering
	\vspace{10pt}
	\begin{tabular}{rrrr}
  		Pair & $\overline{M_{Q^{(z)}}}$ & $\overline{S_{Q^{(z)}}}$ & $d_{Y_\omega}(p, y_z)$ \\ \hline \hline
		\vspace{2pt}
  		$Q_{0,1}$ &  $4.0E-2$ & $1.80E+0$ & $4.84E-1$\\
		\vspace{2pt}
		$Q_{0,2}$ & $1.6E-2$ & $1.02E+0$ & $1.57E-1$\\
		\vspace{2pt}
		$Q_{1,2}$ & $4.7E-2$  & $1.08E+0$ & $1.20E-1$\\
	\end{tabular}
	\captionof{table}{\it $\overline{M_{Q^{(z)}}}$, $\overline{S_{Q^{(z)}}}$, and $d_{Y_\omega}(p, y_z)$ for each of the three possible pairs of QoI.}
	\label{Table:linear_2to3}
\end{minipage}
\end{figure}

\subsubsection{Nonlinear}\label{Sec:Nonlinear}

\begin{figure}
	\centering
	\begin{minipage}[t]{\textwidth}
	\centering
	\vspace{0pt}
        \includegraphics[width=.24\textwidth]{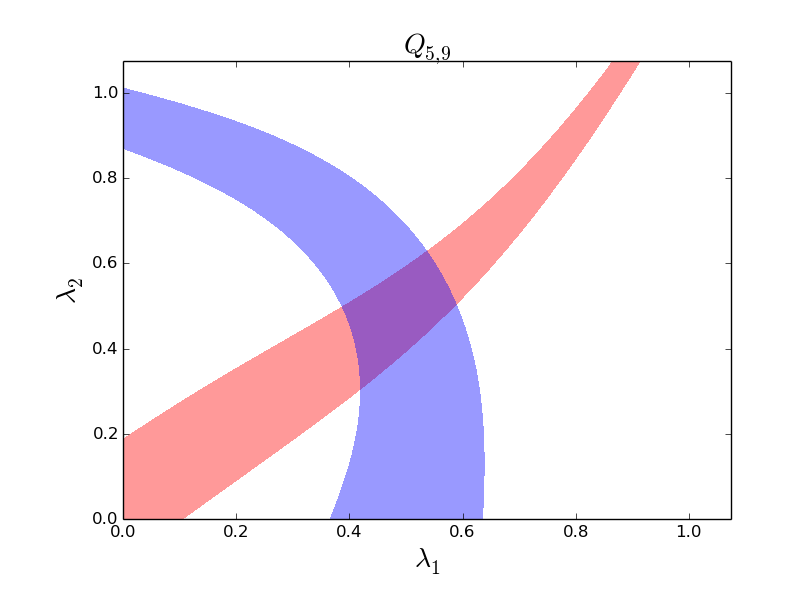}
        		\includegraphics[width=.24\textwidth]{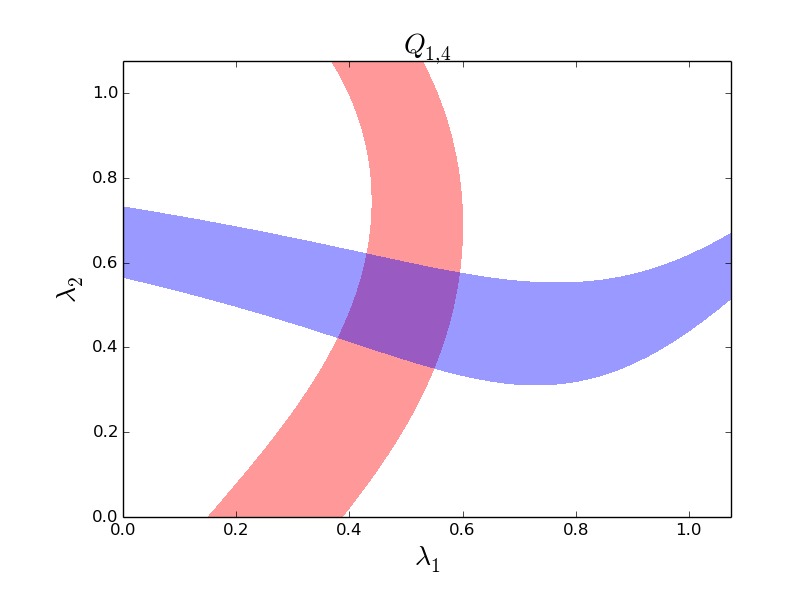}
        			\includegraphics[width=.24\textwidth]{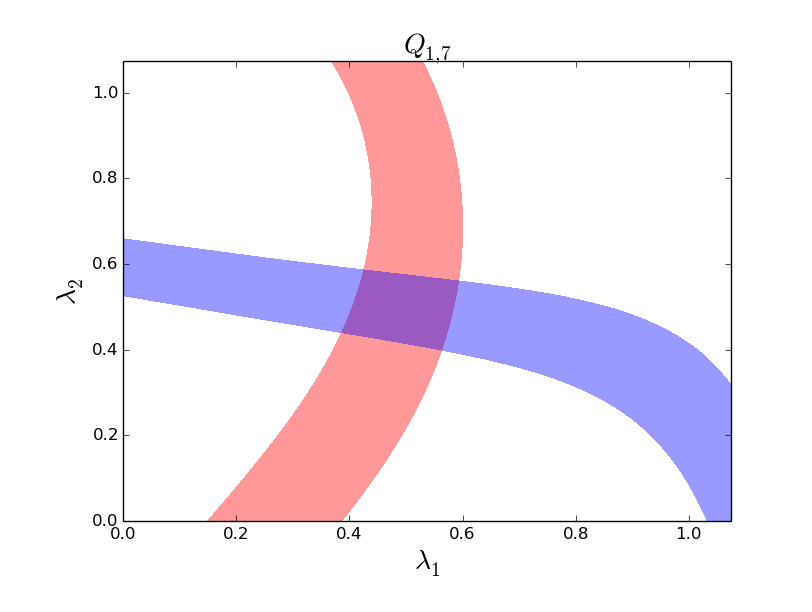}
        				\includegraphics[width=.24\textwidth]{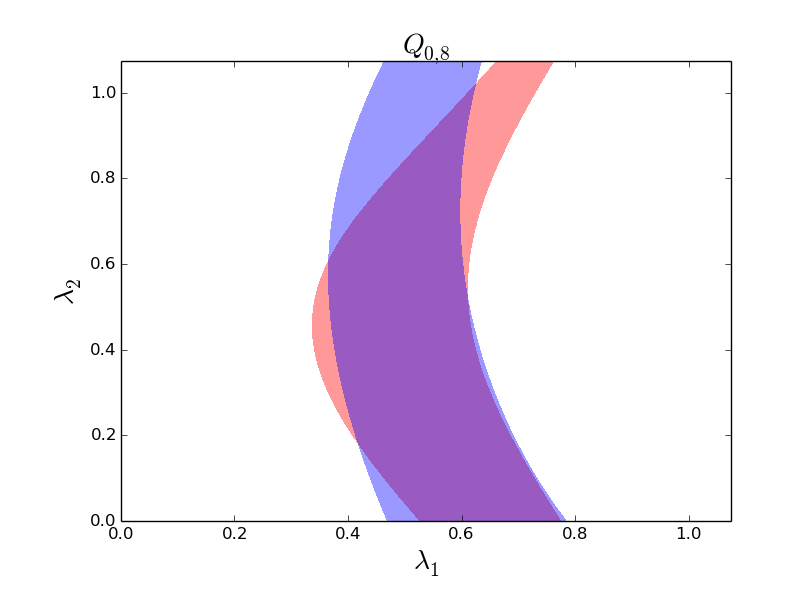}\\
        				
        	\includegraphics[width=.24\textwidth]{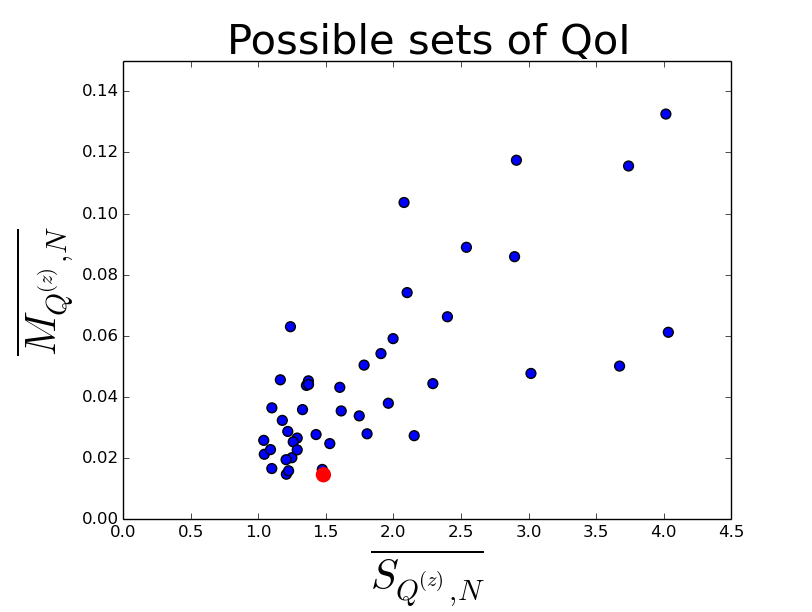}
			\includegraphics[width=.24\textwidth]{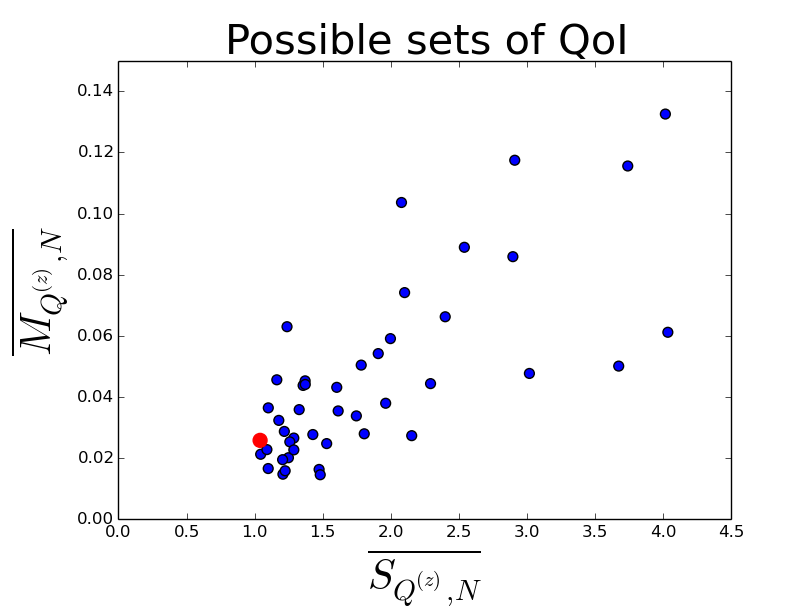}
        			\includegraphics[width=.24\textwidth]{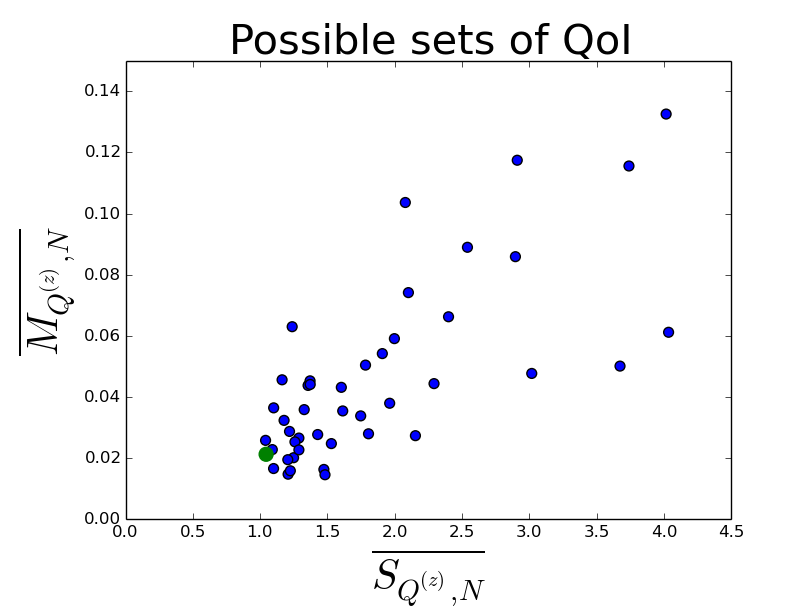}
        				\includegraphics[width=.24\textwidth]{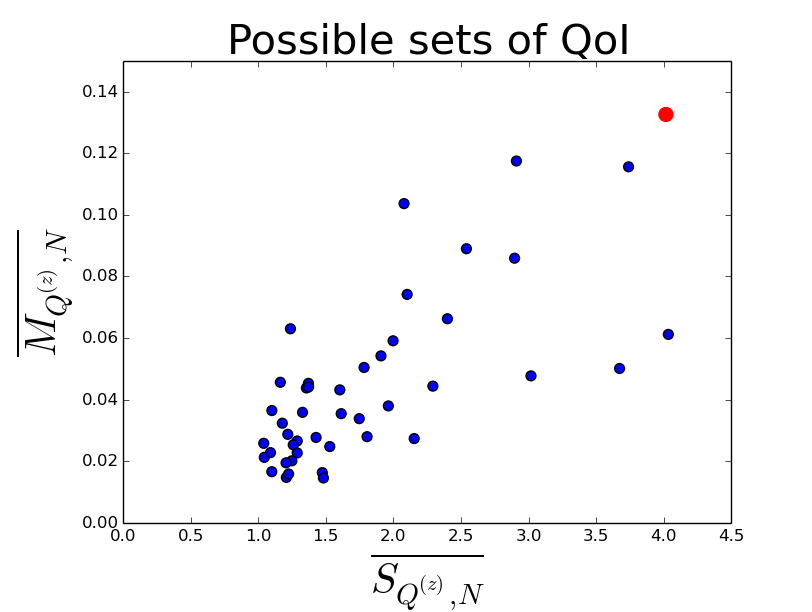}
	\caption{\it On the top we show the inverse image defined by four possible pairs of QoI.  On the bottom we show the point in $Y_\omega$ defined by each pair of QoI.  (left): The pair of QoI that minimizes $\overline{M_{Q^{(z)},N}}$ ($Q_{5,9}$).  (middle-left): The pair of QoI that minimizes $\overline{S_{Q^{(z)},N}}$ ($Q_{1,4}$).  (middle-right): The pair of QoI that minimizes the distance defined in Eq.(\ref{Eq:dist_suppskew}) with $\omega=0.5$ ($Q_{1,7}$).  (right): The pair of QoI that maximizes the distance defined in Eq.(\ref{Eq:dist_suppskew}) with $\omega=0.5$ ($Q_{0,8}$).}\label{Fig:nonlinear_2to10}
	\end{minipage}
	\begin{minipage}[t]{\textwidth}
	\centering
	\vspace{10pt}
	\begin{tabular}{rrrr}
  		Pair & $\overline{M_{Q^{(z)},N}}$ & $\overline{S_{Q^{(z)},N}}$ & $d_{Y_\omega}(p, y_z)$ \\ \hline \hline
		\vspace{2pt}
  		$Q_{5,9}$ &  $1.50E-2$ & $1.482E+0$ & $3.40E-1$\\
		\vspace{2pt}
		$Q_{1,4}$ & $2.60E-2$ & $1.040E+0$ & $6.40E-2$\\
		\vspace{2pt}
		$Q_{1,7}$ & $2.10E-2$  & $1.044E+0$ & $6.30E-2$\\
		\vspace{2pt}
		$Q_{0,8}$ & $1.33E-1$  & $4.016E+0$ & $8.68E-1$\\
	\end{tabular}
	\captionof{table}{\it $\overline{M_{Q^{(z)},N}}$, $\overline{S_{Q^{(z)},N}}$, and $d_{Y_\omega}(p, y_z)$ for the four pairs of QoI considered in Figure~\ref{Fig:nonlinear_2to10}.}
	\label{Table:nonlinear_2to10}
\end{minipage}
\end{figure}

Let $Q : \reals^2 \goto \reals^{10}$ be {\em nonlinear} where each component is a polynomial of the form
\begin{eqnarray*}
	Q_j = r_{j,0}\lambda_1^5 + r_{j,1}\lambda_2^3 + r_{j,2}\lambda_1^3\lambda_2 + r_{j,3}\lambda_1 + r_{j,4}\lambda_2 + r_{j,5} \, \, \text{for}\, \, j=0,1,\hdots,9,
\end{eqnarray*}
where the coefficients $r_{j,k}$ are fixed random numbers in $[-1,1]$.  Let $B^{(z)}\in \mathcal{B}_{\mathcal{D}^{(z)}}$ be the rectangle of uncertainty in the data space defined by the $z^{th}$ possible pair of QoI.  
Again, we assume the uncertainty in each QoI is the same so that $\mu(B^{(z)})$ is constant for all $z$.  
Now that $Q$ is nonlinear, we approximate both the local Jacobians $J_{\lambda^{(i)}}$ and the integrals in Eqs.~\eqref{Eq:suppint} and \eqref{Eq:skewint}.  
We take 1000 uniform random samples $\lambda^{(i)}$ in $\Lambda$ and then compute the approximate Jacobian at $N=100$ of these samples using a Radial Basis Function (RBF) interpolation method.  
Specifically, for a given $\lambda^{(i)}\in\Lambda$, we use the 20 nearest neighbors of $\lambda^{(i)}$ (from the original 1000 samples) to approximate the gradient for each QoI and subsequently construct the Jacobian.

\begin{remark}
	Because $Q$ is a vector valued polynomial function, we could analytically determine the exact gradient vectors for each component and simply evalute at each $\lambda^{(i)}$.  However, the RBF method (or any finite difference method) is a more general approach that can be used to approximate gradient vectors for maps $Q$ defined by functionals of solutions to partial differential equations where we cannot analytically determine the gradients of each component.
\end{remark}

We illustrate the various parts of the optimization problem by separately solving four optimization problems; minimize $\overline{{M_{Q^{(z)},N}}}$, minimize $\overline{{S_{Q^{(z)},N}}}$, minimize Equation~\ref{Eq:dist_suppskew} with $\omega= 0.5$, and maximize Equation~\ref{Eq:dist_suppskew} with $\omega=0.5$.  
The results are summarized in Figure~\ref{Fig:nonlinear_2to10} and Table~\ref{Table:nonlinear_2to10}.  
In the bottom left of Figure~\ref{Fig:nonlinear_2to10} we see the location in $Y_\omega$ of $s_{5,9}=(\overline{S_{Q_{5,9},N}}, \overline{M_{Q_{5,9},N}})$.  Although this pair of QoI minimizes $\overline{M_{Q^{(z)},N}}$ it is obvious from the scatter plot there are pairs of QoI that have similar average $\mu_\Lambda$-measure that have much smaller average skewness.  
In the bottom right we see the location in $Y_\omega$ of $s_{0,8}$ that corresponds to the pair of QoI that maximizes the sum of the average $\mu_\Lambda$-measure and the average skewness.  
The corresponding inverse image, seen in the top right of Figure~\ref{Fig:nonlinear_2to10}, clearly has both the largest measure of support and the largest skewness of the four pairs considered.

\section{\bf Maps Defined by the Solution of a Physics-Based Model}\label{Sec:Map_Defined}

In this section we consider the inverse density of sets under functions defined by the solutions to partial differential equations.  
We first consider a simple time dependent diffusion model with uncertain diffusion coefficient that allows us to understand the results intuitively.
Then we consider the ADvaned CIRCulation model for oceanic, coastal, and estuarine waters (ADCIRC) which incorporates a spatially varying bottom friction model.  
The ADCIRC model depends on many parameter fields, we focus on the bottom friction parameter due to its inherent uncertainty.

\subsection{Time Dependent Diffusion}\label{Sec:heatplate}

\subsubsection{The Model}\label{Sec:The_Model}

We consider the time dependent diffusion equation
\vskip 10pt
\begin{equation*}
    \rho c\frac{\partial T}{\partial t} = \nabla \cdot (\kappa \nabla T) + f, \hskip 10pt x \in \Omega,  t \in (t_0, t_f)
\end{equation*}
\vskip 10pt
\begin{eqnarray*}
    f(x) = 
    \begin{cases}
    Ae^{-\frac{(x_0-p_0)^2+(x_1-p_1)^2}{w}} & \text{if } t \leq t_{source} \\
    0       & \text{if } t > t_{source}
    \end{cases}
\end{eqnarray*}
\begin{eqnarray*}
    T(x;0) = 0 \\
    \frac{\partial T}{\partial n} = 0 ,  x \in \partial \Omega
\end{eqnarray*}
where $\Omega = [-\frac{1}{2}, \frac{1}{2}] \times [-\frac{1}{2}, \frac{1}{2}]$ represents a thin plate made of some alloys, $\rho$ is the density of the plate, $c$ is the heat capacity, and $\kappa$ is the thermal conductivity.  The forcing function $f$ represents an external source with the following source parameters: $A$ is the amplitude, $\bar p$ is the position, and $w$ is the width.  The external source is positioned at the center of the plate.  We remove the external source at time $t_{source}=\frac{t_f}{2}$.  The homogeneous Neumann boundary conditions model perfect insulation around the boundary.

We assume that the square plates are manufactured by welding together two rectangular plates of the same alloy.
However, due to imperfections in the original manufacturing process, we assume that the alloy composition varies significantly from plate to plate.
Thus, we let the thermal conductivity $\kappa$ vary in space as described below. 
We assume that the left-half of $\Omega$ contains an alloy with thermal conductivity $\kappa_0$ and the right-half of $\Omega$ contains an alloy with thermal conductivity $\kappa_1$ on the right, both of which are uncertain within $[0.01, 0.2]$..  
Thus, the model parameters for this problem are $\kappa_0$ and $\kappa_1$ and we have defined our parameter space, $\Lambda = [0.01,0.2] \times [0.01,0.2] \subset \mathbb{R}^2$.
See Figure~\ref{Fig:qoi_locations} where the green and purple colors indicate the two rectangular plates that are welded together to form the single plate defining the physical domain $\Omega$

\subsubsection{The Inverse Problem}\label{Sec:The_Inverse_Problem_heatplate}

Our goal is to determine which pair of QoI produce the best solution to the inverse problem, i.e., determine which two points in space time we should record temperature measurements in order to obtain the ``best'' inferences about the thermal conductivities of each side of the plate.  
Intuitively, we expect to take one temperature measurement on each side of the plate, but it is not clear at what time steps each temperature measurement should be recorded or exactly where we should place the sensors on each side of the plate.  
We limit ourselves to a discrete set of possible temperature measurements in space-time.  
We choose 20 points on the plate in space, see Figure \ref{Fig:qoi_locations}, and gather these temperature measurements at each of 20 time steps.  
Thus, $\mathcal{Q}$ is defned by 400 possible QoI values in space-time for which it is only possible to record two.

To make inferences about the optimal experimental design we must solve the model for a set of samples in the parameter space.  
For the numerical solution of the model,  we use the state of the art open source finite element software FEniCS \cite{AlnaesBlechta2015a} with a $40 \times 40$  triangular mesh, piecewise linear finite elements and a Crank Nicolson time stepping scheme.  We take $5000$ uniform random samples from $\Lambda$ and run the simulation for each sample.  
We mathematically model the sensors recording data at a point in space-time as a functional of the solution by computing the average temperature on a small disc about a point in space but at a fixed time, i.e., each QoI is modeled as
\begin{eqnarray*}
    Q_{(p_i;t_i)} = \frac{1}{\mu_{\Omega}(\mathbf{B}_r(p_i))}\int_{\Omega}T(x;t_i)\chi_{\mathbf{B}_r(p_i)} \, dx. \\
\end{eqnarray*}
This yields a map $Q: \Lambda \to \mathcal{D}$ where $\Lambda$ is 2-dimensional and $\mathcal{D}$ is 400-dimensional.  
Thus, we have $400 \choose 2$ $=79,800$ possible pairs of QoI to consider.
\begin{figure}
    \begin{center}
        \textbf{QoI Locations}\par\medskip
        \includegraphics[width=.53\textwidth]{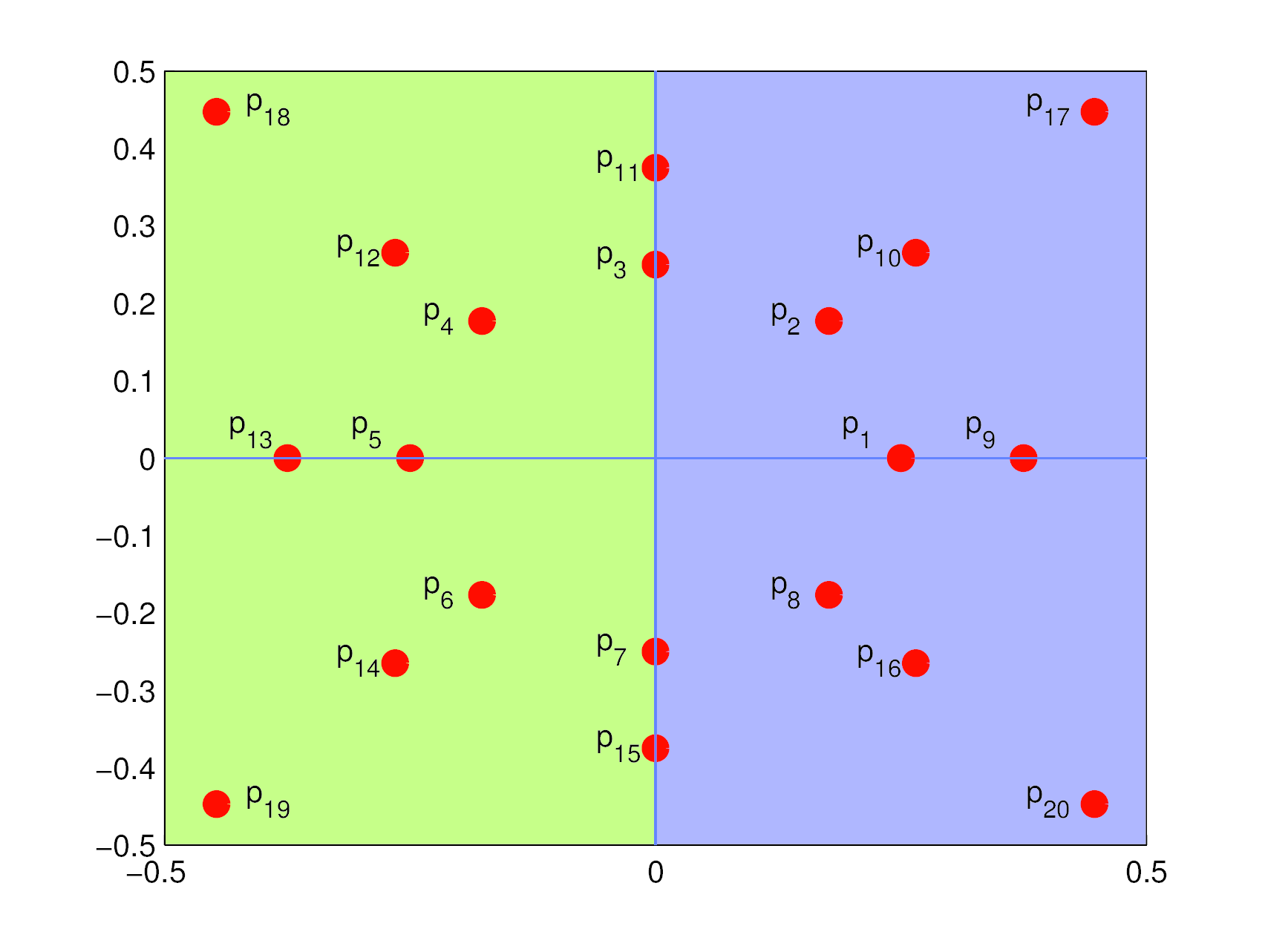}
        \caption{\it The QoI locations in the domain $\Omega$.}\label{Fig:qoi_locations}
    \end{center}
\end{figure}

We approximate $\overline{M_{Q^{(z)}}}$ and $\overline{S_{Q^{(z)}}}$, for each $z$, using the $N=5000$ samples for which we solved the model.  We solve three optimization problems; minimize $\overline{M_{Q^{(z)},N}}$, minimize $\overline{S_{Q^{(z)},N}}$, and minimize the distance defined in Eq.~(\ref{Eq:dist_suppskew}) with $\omega=0.5$.  
Recall these minimization problems tell us {\em on average} the best possible set of QoI to use in the inverse problem before any physical experiment occurs.

The results of the optimization problems are shown in Table~\ref{Table:heatplate}.  
As expected, each pair of QoI found has one point on the left side of the plate and one point on the right.  
However, there are notable differences in space-time locations of each optimal pair of QoI.  
The pair of QoI that minimizes $\overline{M_{Q^{(z)},N}}$ is produced by measurements made near the end of the simulation.  
The pair of QoI that minimizes $\overline{S_{Q^{(z)},N}}$ is produced by measurements made at the first time step. 
 Intuitively, this makes sense because at $t_1$ the temperature measurements are primarily determined by local alloy properties, i.e., not enough time has passed for the left side of the plate to influence the temperature measurements on the right side of the plate and vice versa.

We solve the inverse problem three times, once for each pair of QoI described above.  
We simulate real data, i.e., a temperature measurement made on a given plate with uncertain thermal conductivities in a laboratory experiement, by mapping a random point $\lambda_{ref}\in\Lambda$ forward to $Q^{(z)}(\lambda_{ref})=q^{(z)}_{ref}\in\mathcal{D}^{(z)}$.  
We assume the uncertainty in this {\em observed} datum is the same for each QoI and let this uncertainty be $0.5$ degrees, i.e., the 2-dimensional square we invert for each pair of QoI is centered at $q^{(z)}_{ref}$ and has side length of $0.5$.  
We solve the inverse problem using the open source BET \cite{BET} python package for each of the resulting pairs of QoI, see Figure \ref{Fig:heatplate_inverse}.

In the top row we see in purple the inverse image for each pair of QoI approximated with $N=5000$ uniform samples.  
Notice in the top middle, the pair that minimizes the average skewness, the inverse image is the entire parameter space, i.e., we have not reduced the uncertainty in our model input parameters at all.  
Although this inverse image can be well approximated by relatively few samples, it is useless as it gives no new information about the location of the parameters that produced the observed datum.  
In the bottom row we show the location of $y_z\in Y_\omega$ for each pair of QoI.  
In the bottom right, notice the apparent smooth curve being defined by these these points in $Y_\omega$.  
In multicriteria optimization this is referred to as the Pareto frontier, and we see in all three cases our optimal choice of pair of QoI produces a point in $Y_\omega$ that lies very near the Pareto frontier.

\begin{figure}
	\centering
	\begin{minipage}[t]{\textwidth}
	\centering
	\vspace{0pt}
        \includegraphics[width=.3\textwidth]{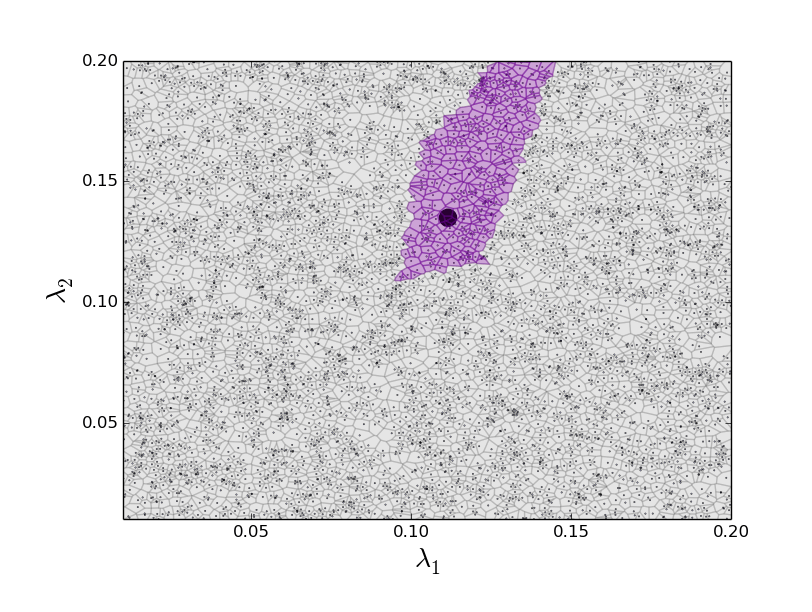}
        		\includegraphics[width=.3\textwidth]{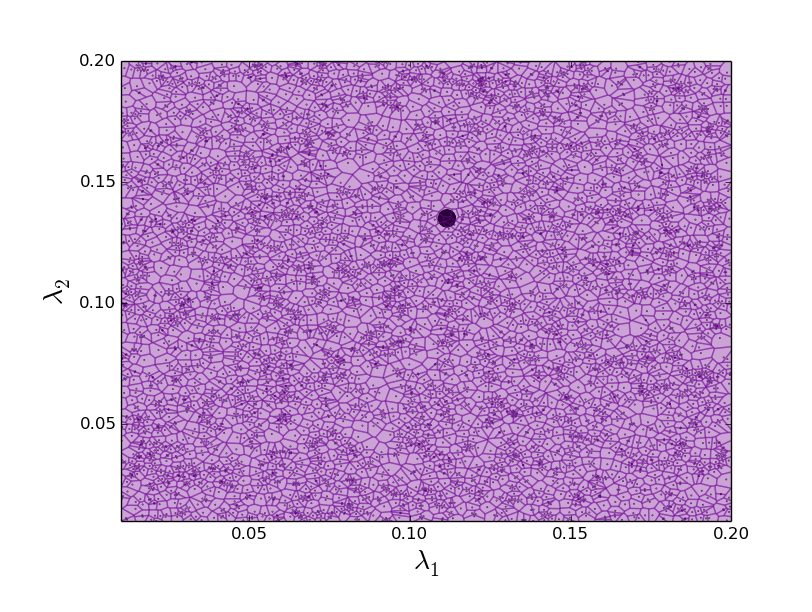}
        				\includegraphics[width=.3\textwidth]{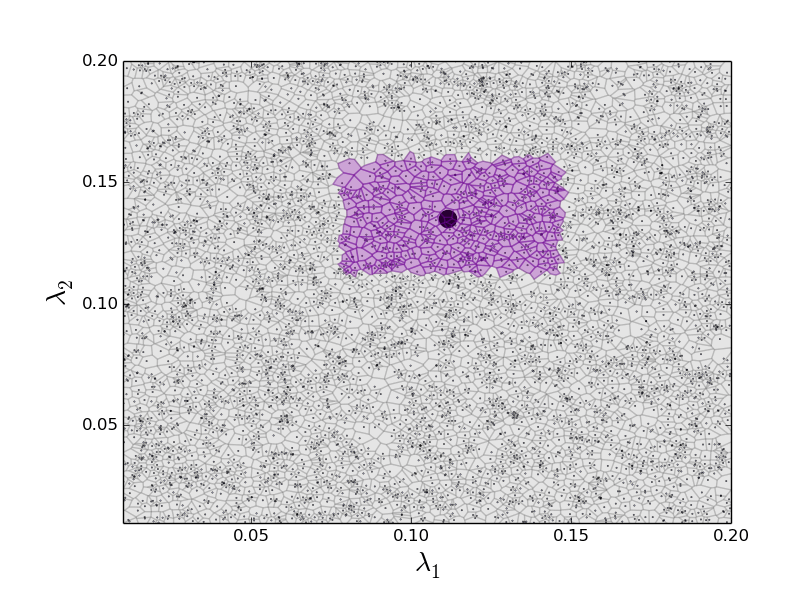}\\

        	\includegraphics[width=.3\textwidth]{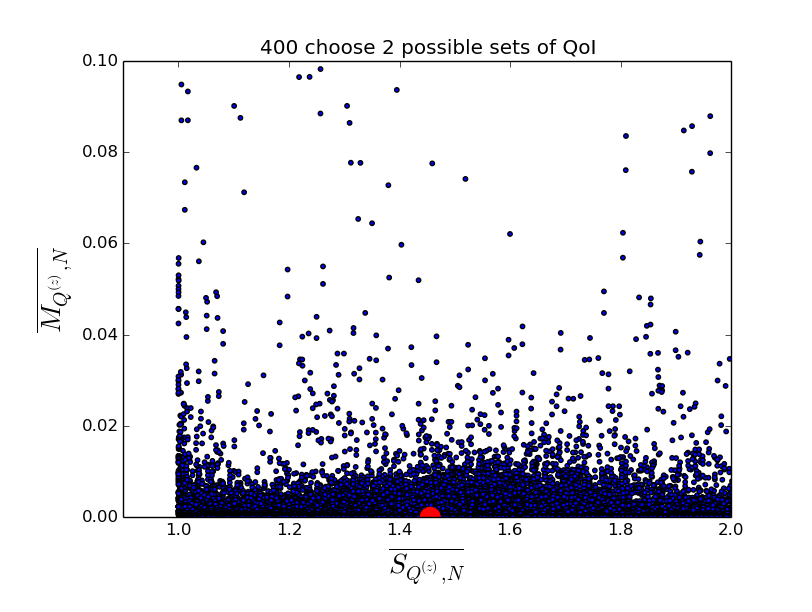}
			\includegraphics[width=.3\textwidth]{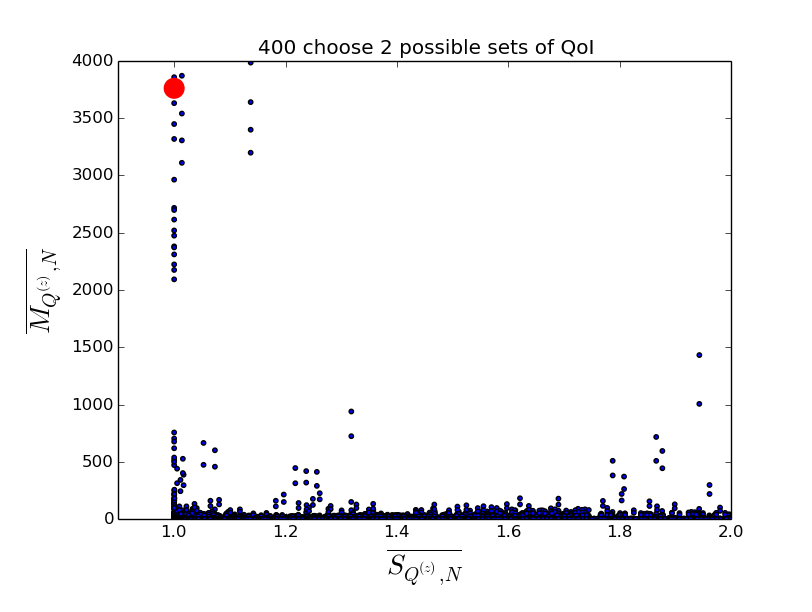}
        			\includegraphics[width=.3\textwidth]{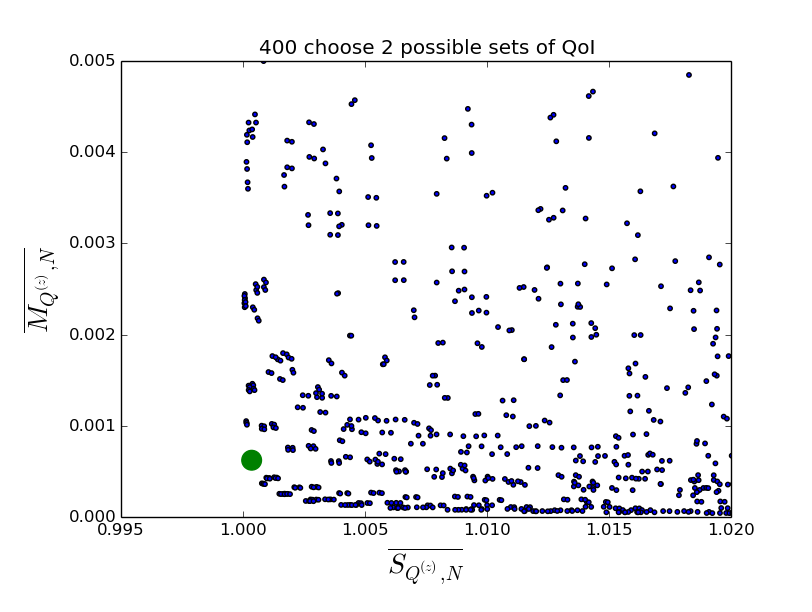}
	\caption{\it (top): The inverse image approximated using the BET package.  (bottom): The space $Y_\omega$ defined in Section \ref{Sec:Optimizing}.  (left): The pair of QoI that minimizes $\overline{M_{Q^{(z)},N}}$.  (middle) : The pair of QoI that minimizes $\overline{S_{Q^{(z)},N}}$.  (right): The pair of QoI that minimizes the distance defined in Eq.(\ref{Eq:dist_suppskew}) with $\omega=0.5$.}\label{Fig:heatplate_inverse}
	\end{minipage}
	\begin{minipage}[t]{\textwidth}
	\centering
	\vspace{10pt}
	\begin{tabular}{rrrr}
	    QoI Locations & $\overline{M_{Q^{(z)},N}}$ & $\overline{S_{Q^{(z)},N}}$ & $d_{Y_\omega}(p, y_z)$ \\ \hline \hline
	
	    $Q_{(p_{18}, t_{14})}, Q_{(p_{17}, t_{15})}$  &  $1.72E-05$ & $1.46E+00$ & $3.31E-01$ \\
	
		$Q_{(p_{1}, t_{1})}, Q_{(p_{12}, t_{1})}$ & $3.76E+03$ & $1.00E+00$ & $9.99E-01$ \\
	
		$Q_{(p_{17}, t_{4})}, Q_{(p_{18}, t_{4})}$ & $6.26E-04$ & $1.0004E+00$ & $9.84E-04$ \\
	
	\end{tabular}
	\captionof{table}{\it $\overline{M_{Q^{(z)},N}}$, $\overline{S_{Q^{(z)},N}}$, and $d_{Y_\omega}(p, y_z)$ for the three pairs of QoI considered in Figure~\ref{Fig:heatplate_inverse}.}
	\label{Table:heatplate}
\end{minipage}
\end{figure}

\subsubsection{The Prediction Problem}\label{Sec:The_Prediction_Problem_heatplate}

As mentioned in Section~\ref{Sec:Precision_Accuracy}, following solution to the stochastic inverse problem, we generally want to make, and quantify uncertainty in, predictions that can be formulated as solutions to a stochastic forward problem.   
Typically these predictions correspond to unobservable QoI, e.g., the maximum storm surge elevation for a future hurricane being forced by currently unknown winds and tides.  
For simplicity, we consider a prediction problem for a similar QoI in the heating of a thin plate as described above.  
We denote the prediction QoI as $Q^{(p)}$ which represents the average temperature along the right side of the plate  at the final time step with the external source now located at the bottom left of the plate.
\begin{eqnarray*}
	Q^{(p)} = \frac{1}{\mu_{\Omega}(E)}\int_{\Omega}T(x;t_{20})\chi_{E} \, dx,\\
    f(x) = 
    \begin{cases}
    Ae^{-\frac{(x_0 + \frac{1}{2})^2 + (x_1 + \frac{1}{2})^2}{w}} & \text{if } t \leq t_{source} \\
    0       & \text{if } t > t_{source}.
    \end{cases}
\end{eqnarray*}
Here, $E$ is a small vertical strip in $\Omega$ along the right boundary.  
For each inverse solution found in Section~\ref{Sec:The_Inverse_Problem_heatplate} we have a corresponding set of samples from the original $N=5000$ that lie inside the inverse density.  
For each of these samples, we run the simulation and evaluate $Q^{(p)}$. 
The range of the results for each of the three inverse images is shown in Table~\ref{Table:prediction}.  
In the first column we give the pair of QoI used to solve the stochastic inverse problem.  
In the second column we show the relative measure of the support of the probability measure solving the stochastic inverse problem.  
In the third column we show the interval of possible prediction values for $Q^{(p)}$.  
We see the pair of QoI that minimizes the sum of the average $\mu_\Lambda$-measure and the average skewness (last row) produces a prediction interval about 5 times smaller than propagating the entire parameter space forward (second row).
Also, this pair of QoI produces a prediction interval of similar precision (as determined by length of the interval) to the prediction interval corresponding to the first row, which we would expect to be the best but perhaps the most computationally complex to use. 

\begin{table}
\centering
\begin{tabular}{ccc}
    QoI Locations & $\frac{\mu_\Lambda(\supp((Q^{(z)})^{-1}(B)))}{\mu_\Lambda(\Lambda)}$ & Prediction Interval  \\ \hline \hline

    $Q_{(p_{18}, t_{14})}, Q_{(p_{17}, t_{15})}$  & $6.68E-2$  & [18.3, 24.0] \\

	$Q_{(p_{1}, t_{1})}, Q_{(p_{12}, t_{1})}$ & $1.00E+0$ & [0.06, 25.3] \\

	$Q_{(p_{17}, t_{4})}, Q_{(p_{18}, t_{4})}$ & $8.62E-2$ & [16.9, 23.1] \\

\end{tabular}
\captionof{table}{\it Predictions made by propagating forward each of the three inverse images in Figure~\ref{Fig:heatplate_inverse}.}
\label{Table:prediction}
\end{table}

\subsection{Coastal Ocean Model}\label{Sec:Costal_Ocean}

\subsubsection{The Model}\label{Sec:The_Model_adcirc}

Coastal ocean models numerically solve the shallow water equations (SWEs), which model the flow of water processes on domains with vertical length scales that are negligible relative to the horizontal length scales.  Integrating out the depth results in a first-order hyperbolic continuity equation coupled to the momentum equations for horizontal depth-averaged velocities given by
\begin{eqnarray}
	\frac{\partial H}{\partial t} + \frac{\partial}{\partial x}(Q_x) + \frac{\partial}{\partial y}(Q_y) &=& 0, \label{Eq:continuity} \\ 
	\frac{\partial Q_y}{\partial t} + \frac{\partial UQ_x}{\partial x} + \frac{\partial VQ_x}{\partial y} - fQ_x &=&\nonumber\\
	-gH\frac{\partial(\zeta + P_s/g\rho_0-\alpha\eta)}{\partial x} &+& \frac{\tau_{sx}}{\rho_0} - \frac{\tau_{bx}}{\rho_0} + M_x - D_x - B_x, \label{Eq:mom1} \\ 
	\frac{\partial Q_x}{\partial t} + \frac{\partial UQ_y}{\partial x} + \frac{\partial VQ_y}{\partial x} - fQ_x &=&\nonumber\\
	-gH\frac{\partial(\zeta + P_s/g\rho_0-\alpha\eta)}{\partial y} &+& \frac{\tau_{sy}}{\rho_0} - \frac{\tau_{by}}{\rho_0} + M_y - D_y - B_y.\label{Eq:mom2}
\end{eqnarray}
Here, $\zeta$ is the free surface departure from the geoid, $h$ is the bottom surface departure from the geoid, $H=\zeta + h$ is the total water column height, $U_i$ is the depth-averaged velocity in the $x_i$ direction, $f$ is the Coriolis parameter, $P_s$ is the atmospheric pressure at the free surface, $\rho_0$ is the reference density of water, $\alpha$ is the effective earth elasticity factor, $\eta$ is the Newtonian equilibrium tide potential, $\tau_{sx_i}$ are the imposed surface stresses, $\tau_{bx_i}$ are the bottom stress components, $M_{x_i}$ is the vertically-integrated lateral stress gradient, $D_{x_i}$ is the momentum dispersion, and $B_{x_i}$ is the vertically-integrated baroclinic pressure gradient.

In the ADvaced CIRCulation (ADCIRC) model the continuity equation, Equation~\ref{Eq:continuity}, is replaced by the second order hyperbolic generalized wave continuity equation (GWCE), Eq.~\eqref{Eq:GWCE} \cite{Butler2012b}.  
Together, the GWCE and the momentum equations, Eqs.~\eqref{Eq:mom1} and \eqref{Eq:mom2}, define the modified form of the SWEs solved by ADCIRC.  
The GWCE and momentum equations are discretized in space by a piecewise linear triangular mesh.  
The time stepping is done with centered finite differences for the GWCE and forward finite difference for the momentum equations.

\begin{eqnarray}
	&&\frac{\partial^2\zeta}{\partial t^2} + \tau_0\frac{\partial \zeta}{\partial t} + \frac{\partial}{\partial x}\bigg(-\frac{\partial UQ_x}{\partial x} - \frac{\partial VQ_x}{\partial x} + fQ_y - gH\frac{\partial(\zeta + P_s/g\rho_0-\alpha\eta)}{\partial x}\nonumber\\
	&&+ \frac{\tau_{sx}}{\rho_0} - \frac{\tau_{bx}}{\rho_0} + M_x - D_x - B_x + \tau_0Q_x\bigg) + \frac{\partial}{\partial y}\bigg(-\frac{\partial UQ_y}{\partial x} - \frac{\partial VQ_y}{\partial x} + fQ_x\nonumber\\
	&&- gH\frac{\partial(\zeta + P_s/g\rho_0-\alpha\eta)}{\partial y}	+ \frac{\tau_{sy}}{\rho_0} - \frac{\tau_{by}}{\rho_0} + M_y - D_y - B_y + \tau_0Q_y\bigg)\nonumber\\
	&&- UH\frac{\partial\tau_0}{\partial x} - VH\frac{\partial\tau_0}{\partial y} = 0.\label{Eq:GWCE}
\end{eqnarray}

A driving application of the ADCIRC model is predicting maximum storm surge elevations during extreme weather events such as hurricanes.  
Accurate estimation of storm surge elevations requires accurate estimation of model parameters.  
The Manning's n coefficients of roughness are of particular importance.  
They enter into the SWEs through the bottom stress components $\tau_{bx_i}$ in the momentum equations.  
The Manning's n coefficient is a highly variable spatial parameter dependent on the surface characteristics of the seabed and is inherently uncertain.

For the purposes of this work, the derivation of how the Manning's n coefficient enters into the momentum equation is not particularly important.  
In the end, as we change the Manning's n coefficient we see changes in the maximum storm surge elevations in the physical domain.

\begin{table}
\centering
\begin{tabular}{ccc}
    Parameter & Range of Manning's n values & Land Classification  \\ \hline \hline

    $\lambda_1$  & $[0.0396, 0.21]$  & Low-intensity developed \\

	$\lambda_2$ & $[0.0594, 0.315]$ & Evergreen forest \\

	$\lambda_3$ & $[0.0495, 0.2625]$ & Palustrine forested wetland \\

	$\lambda_4$ & $[0.00825, 0.04375]$ & Open water \\
\end{tabular}
\captionof{table}{\it Manning’s n coeffcient ranges for the subdomain.}
\label{Table:Mannings}
\end{table}

\subsubsection{The Inverse Problem}\label{Sec:adcirc}

We consider a subdomain of the Golf Coast that has four dominant land classifications and therefore four dominant Manning's n coefficients that enter into the bottom stress components in the momentum equations, see Table~\ref{Table:Mannings}.  
We run ADCIRC to simulate the storm surge under the forcing of Hurricane Gustav and record maximum storm surge elevations at 194 spatial locations, see Figure~\ref{Fig:subdomain}.  
Our goal is to determine which set of four spatial locations provides us with data that produces the best solution to the inverse problem.

\begin{figure}
    \begin{center}
        \includegraphics[width=.45\textwidth]{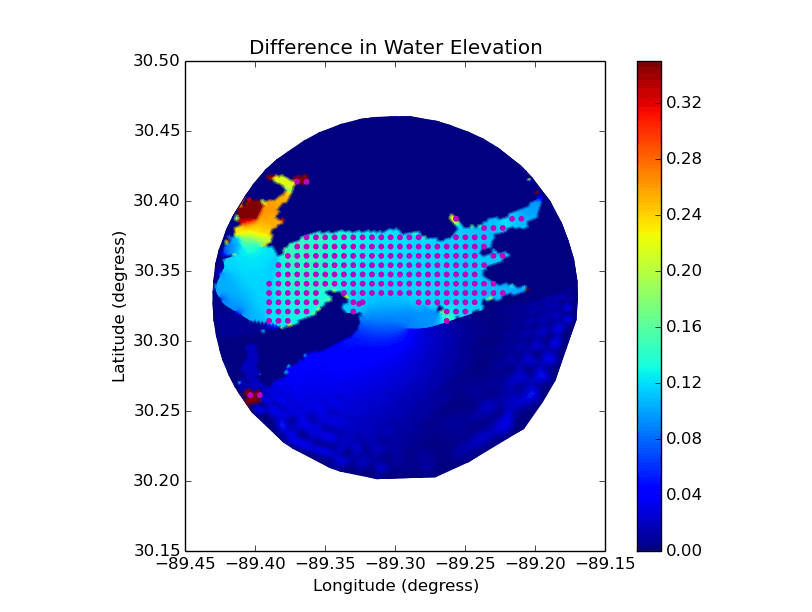}\\
        		\includegraphics[width=.45\textwidth]{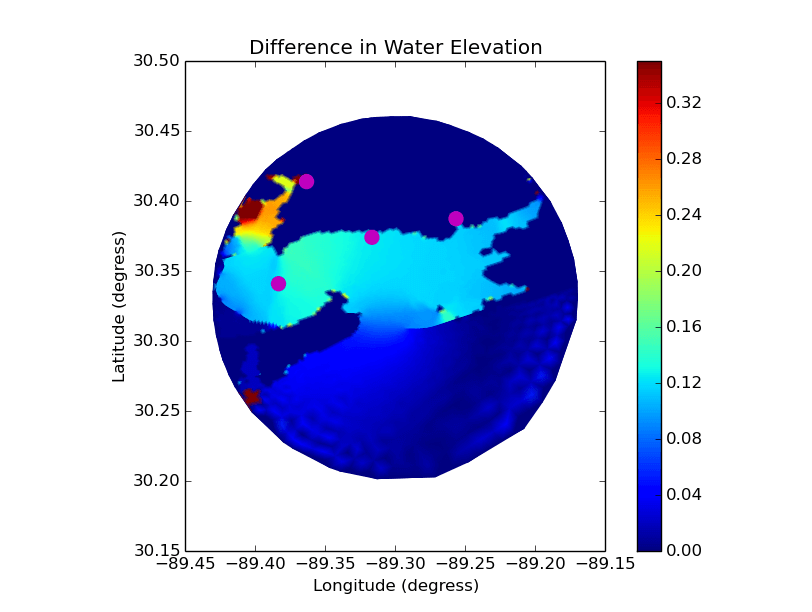}
        			\includegraphics[width=.45\textwidth]{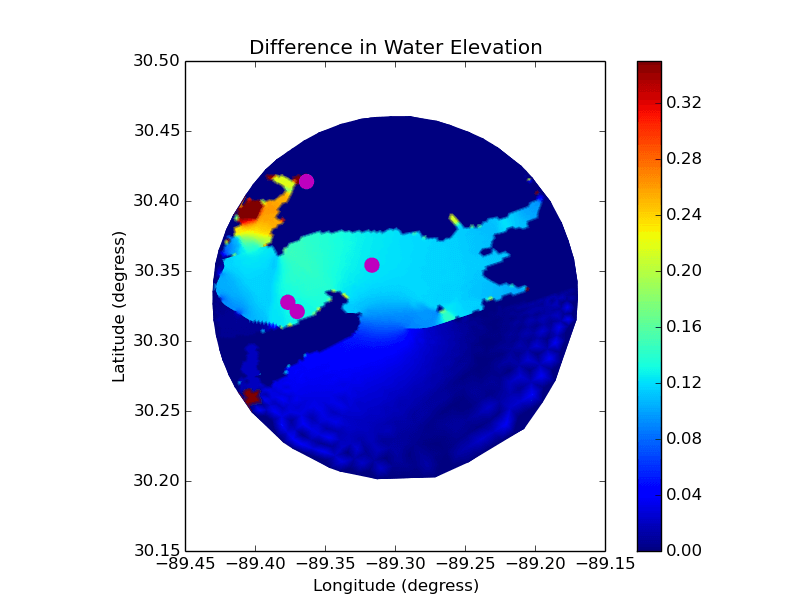}
        \caption{\it (top): 194 potential station locations.  (bottom left): Stations [11, 40, 160, 91] define the optimal observation network.  (bottom right): Stations [16, 40, 22, 88] define a suboptimal observation network.}\label{Fig:subdomain}
    \end{center}
\end{figure}

\begin{figure}
    \begin{center}
        \includegraphics[width=.24\textwidth]{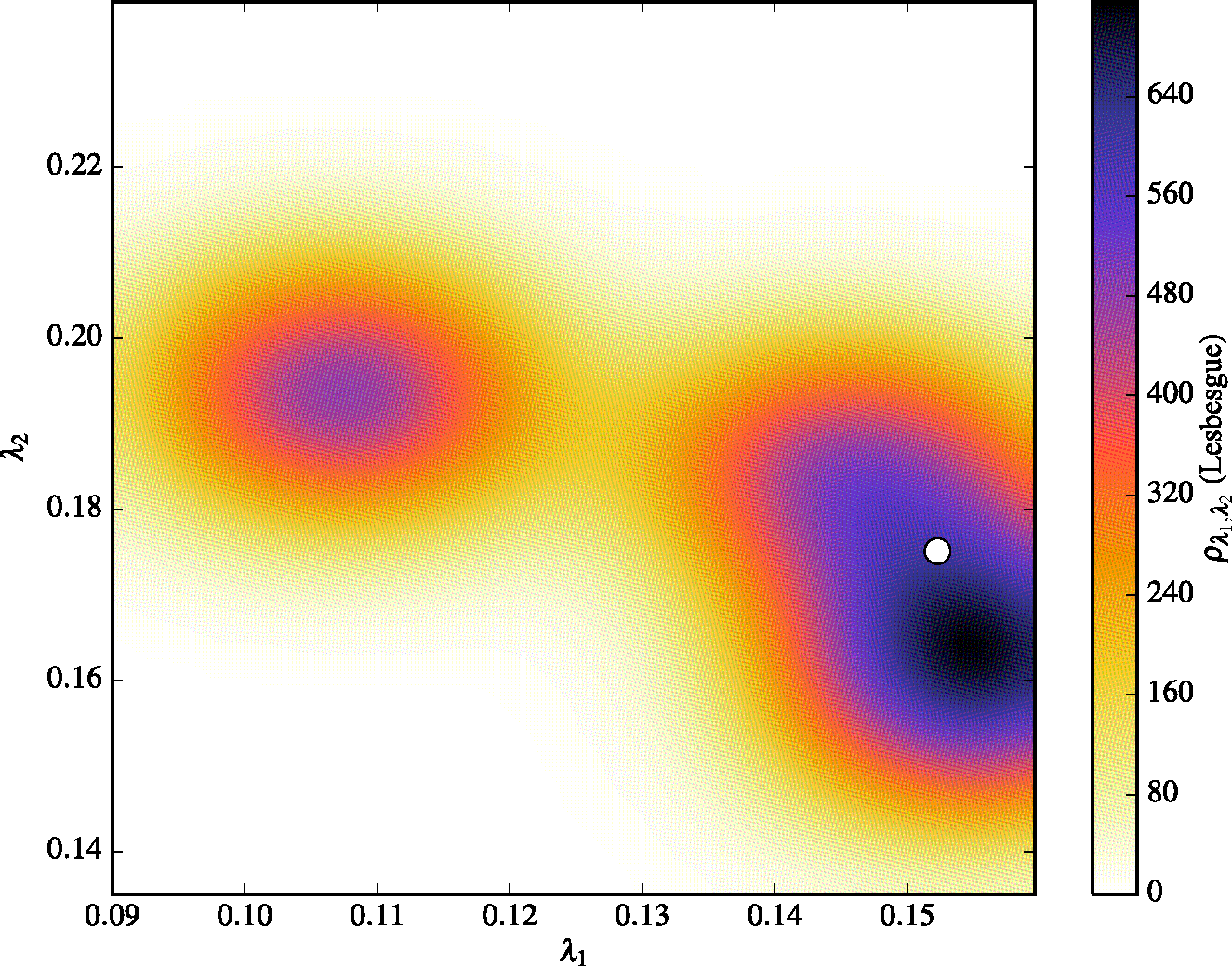}
        		\includegraphics[width=.24\textwidth]{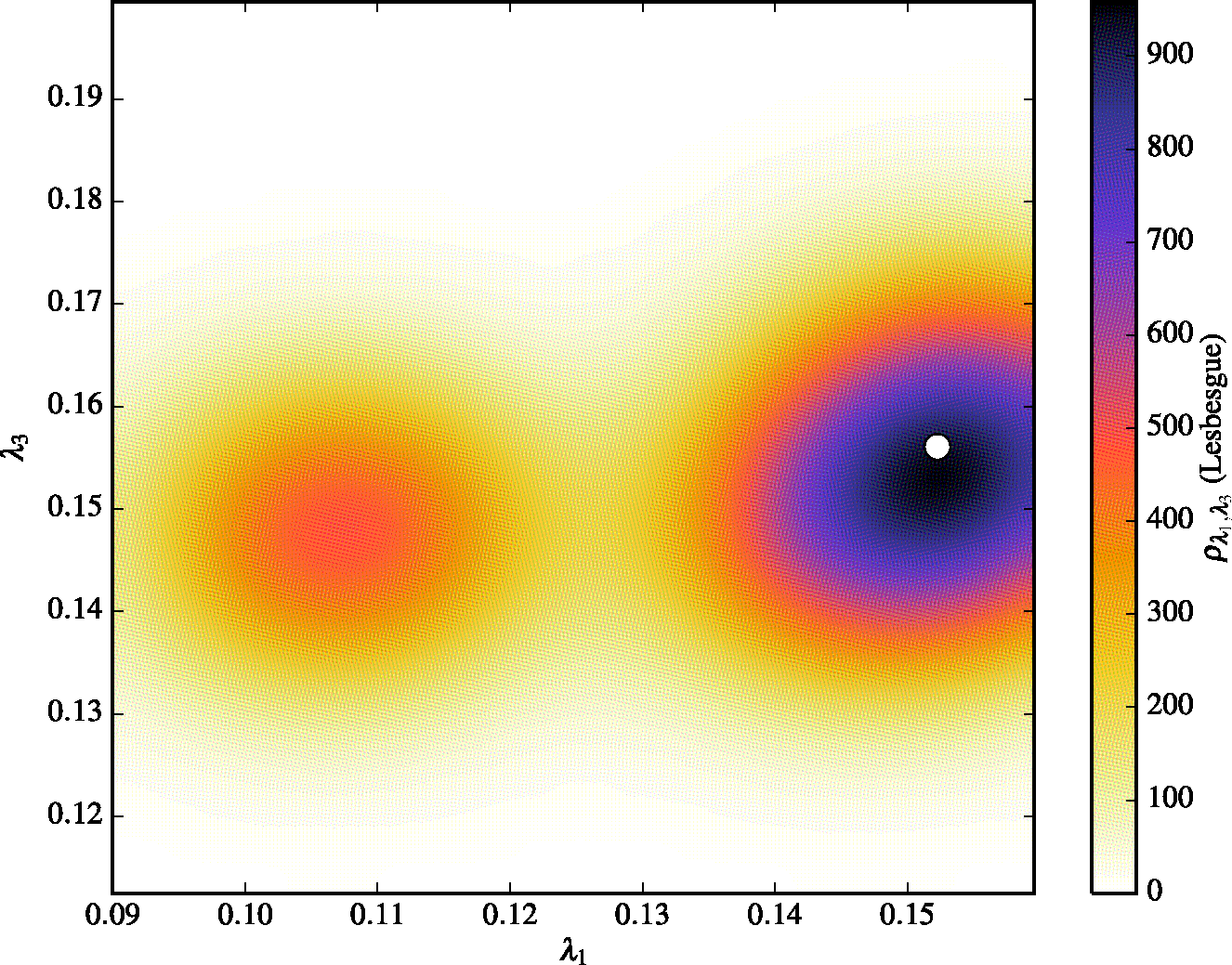}
        			\includegraphics[width=.24\textwidth]{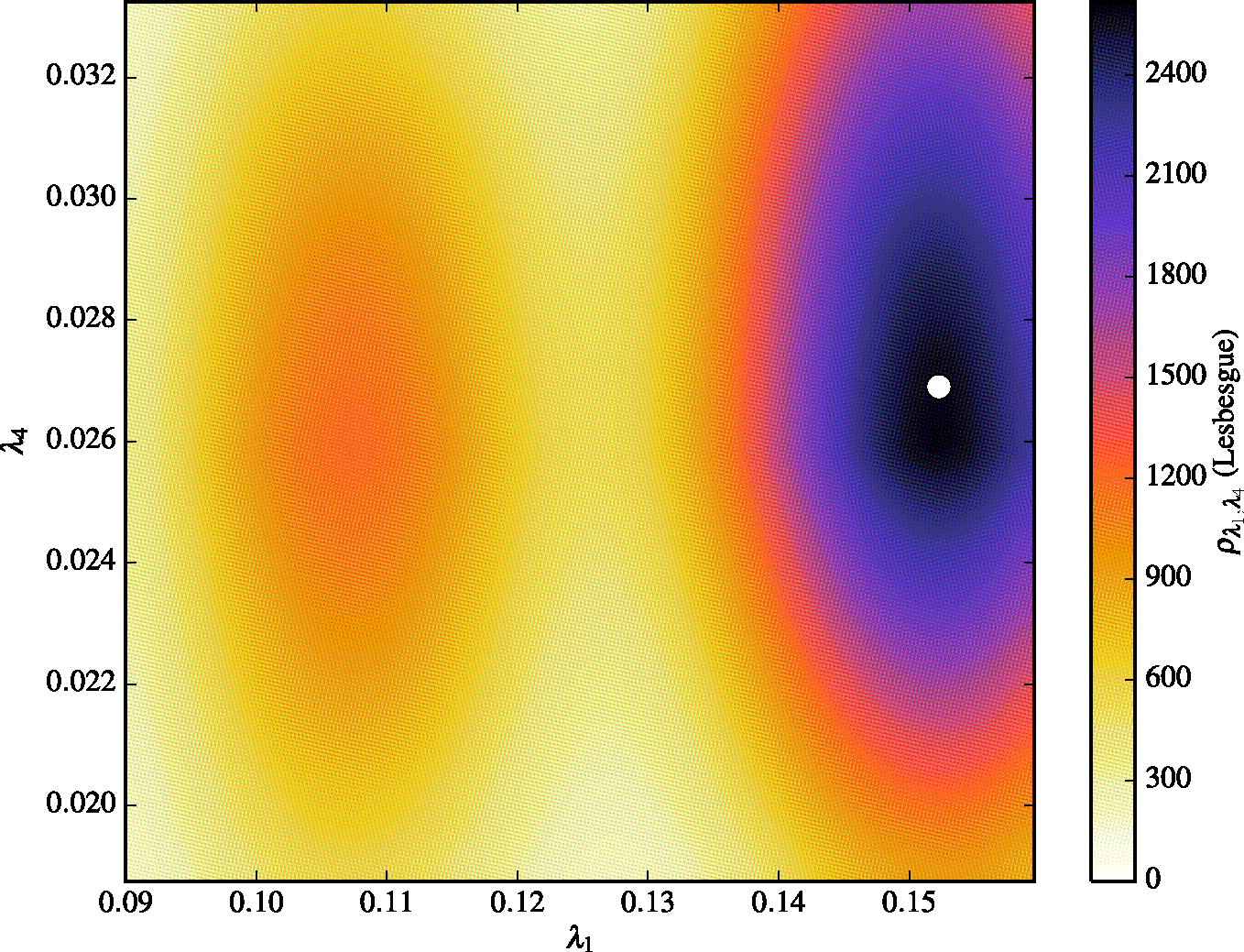} \\
		\includegraphics[width=.24\textwidth]{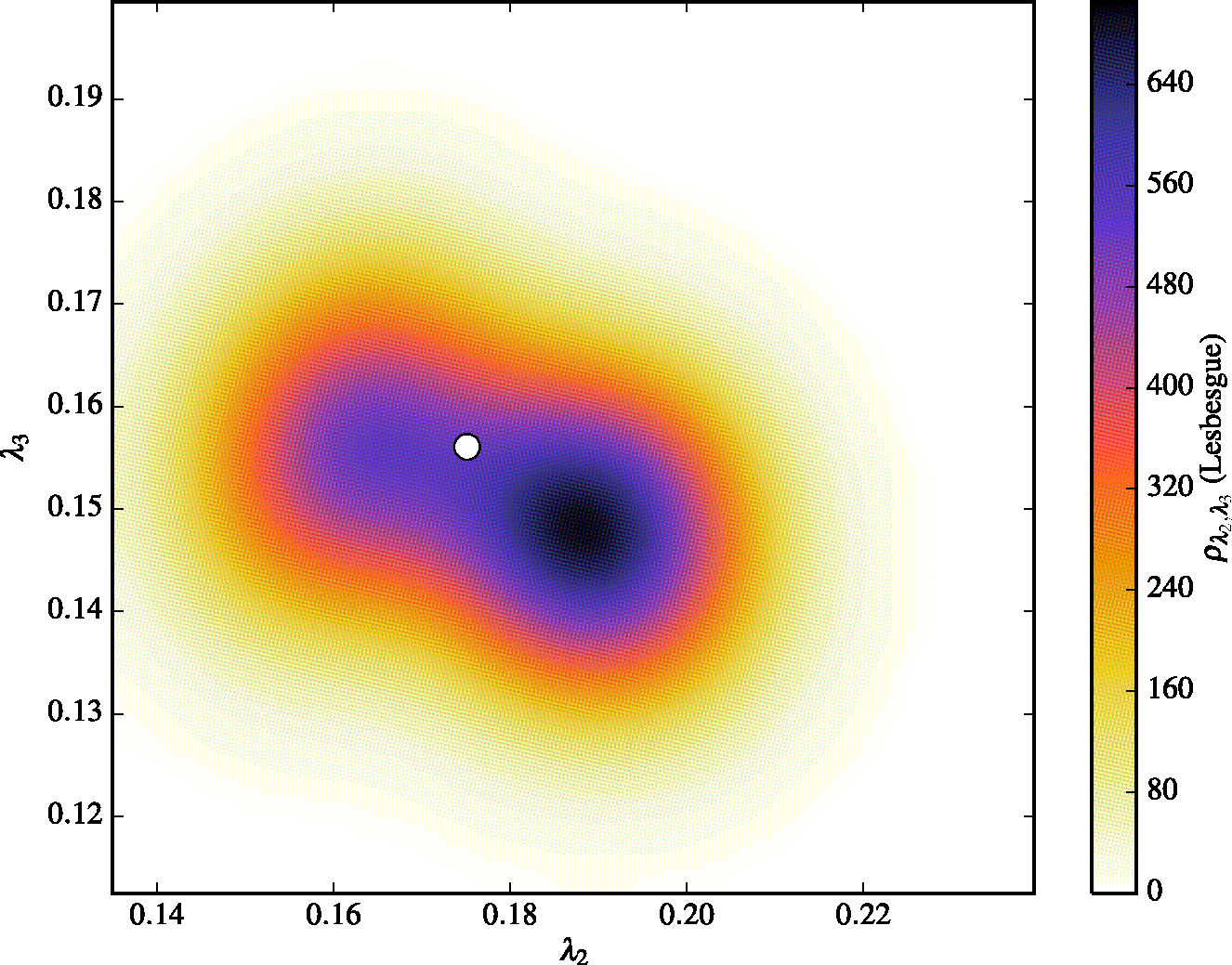}
        		\includegraphics[width=.24\textwidth]{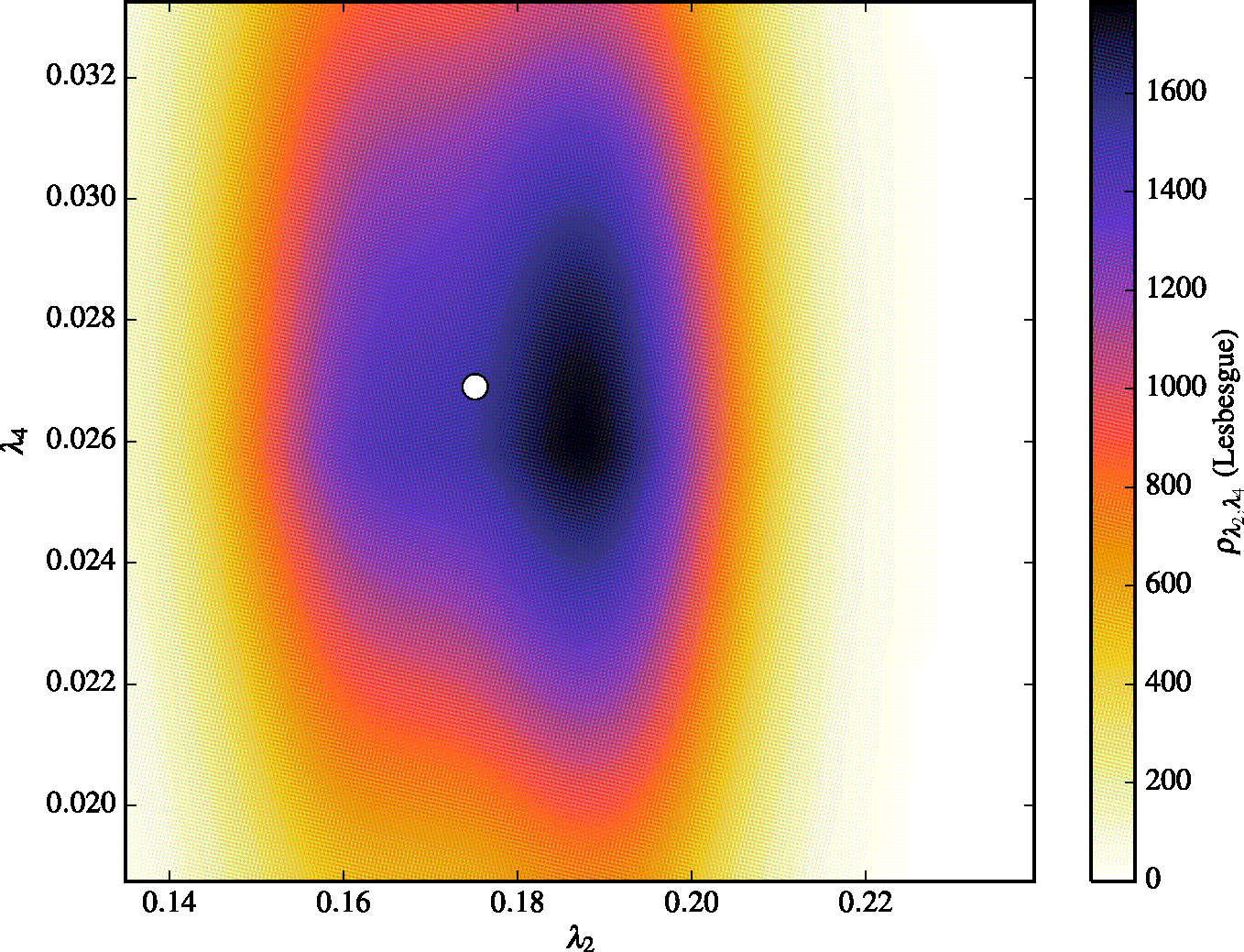}
				\includegraphics[width=.24\textwidth]{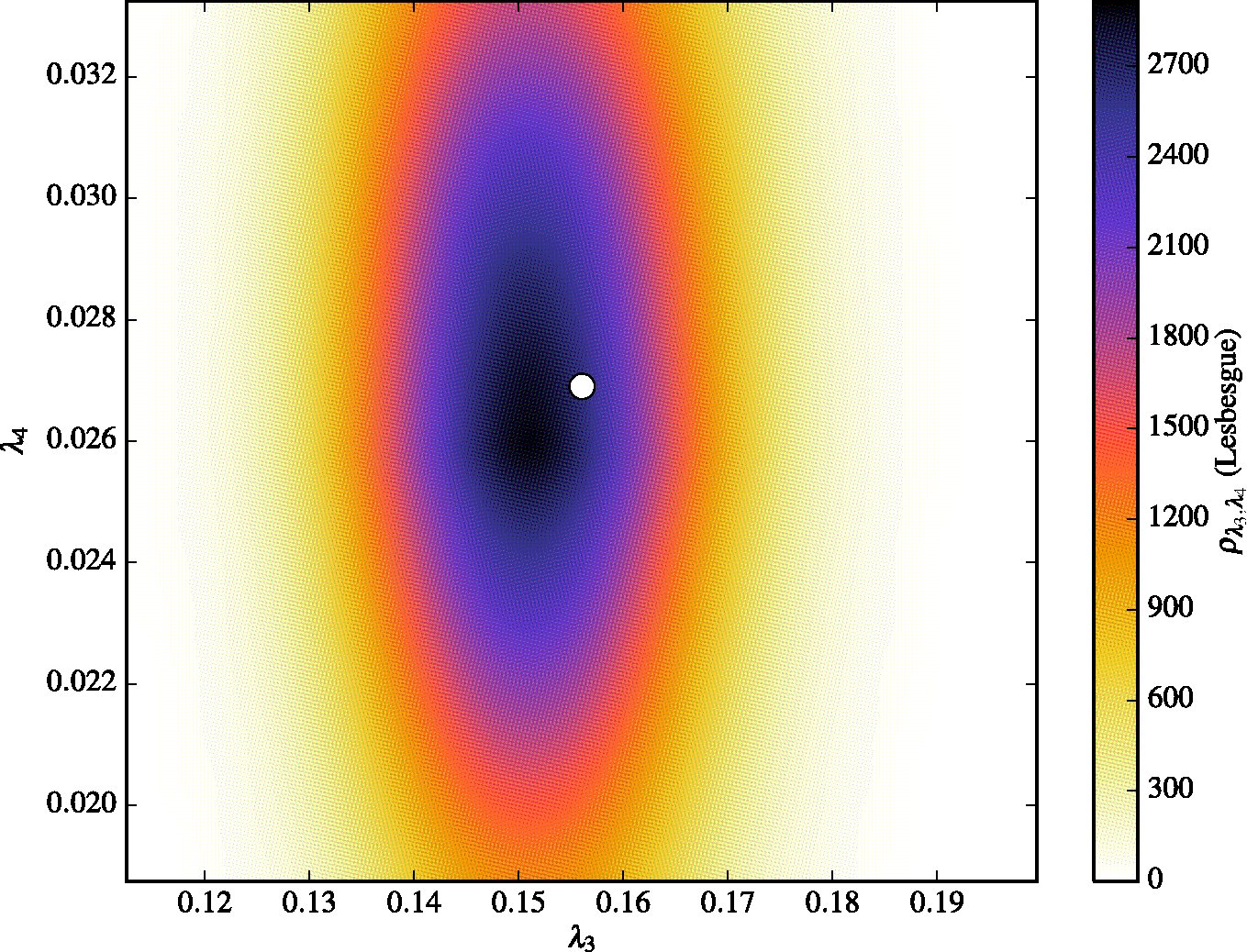}
        \caption{\it Plots of the marginals of the inverse solution using uniform samples for optimal stations [11, 40, 160, 191]. Here, $\rho_{\mathcal{D}}$ is defined as a uniform density on a small rectangular box centered at the reference QoI values associated with $\lambda_{ref} = (0.1523, 0.1751, 0.1561, 0.0269)$. The reference value is illustrated by a white circle. (left): In order from top to bottom $(\lambda_1, \lambda_2),(\lambda_1, \lambda_4),(\lambda_2, \lambda_4)$. (right): In order from top to bottom $(\lambda_1, \lambda_3),(\lambda_2, \lambda_3),(\lambda_3, \lambda_4)$.}\label{Fig:adcirc_optimal}
        \textbf{}\par\medskip
        \includegraphics[width=.24\textwidth]{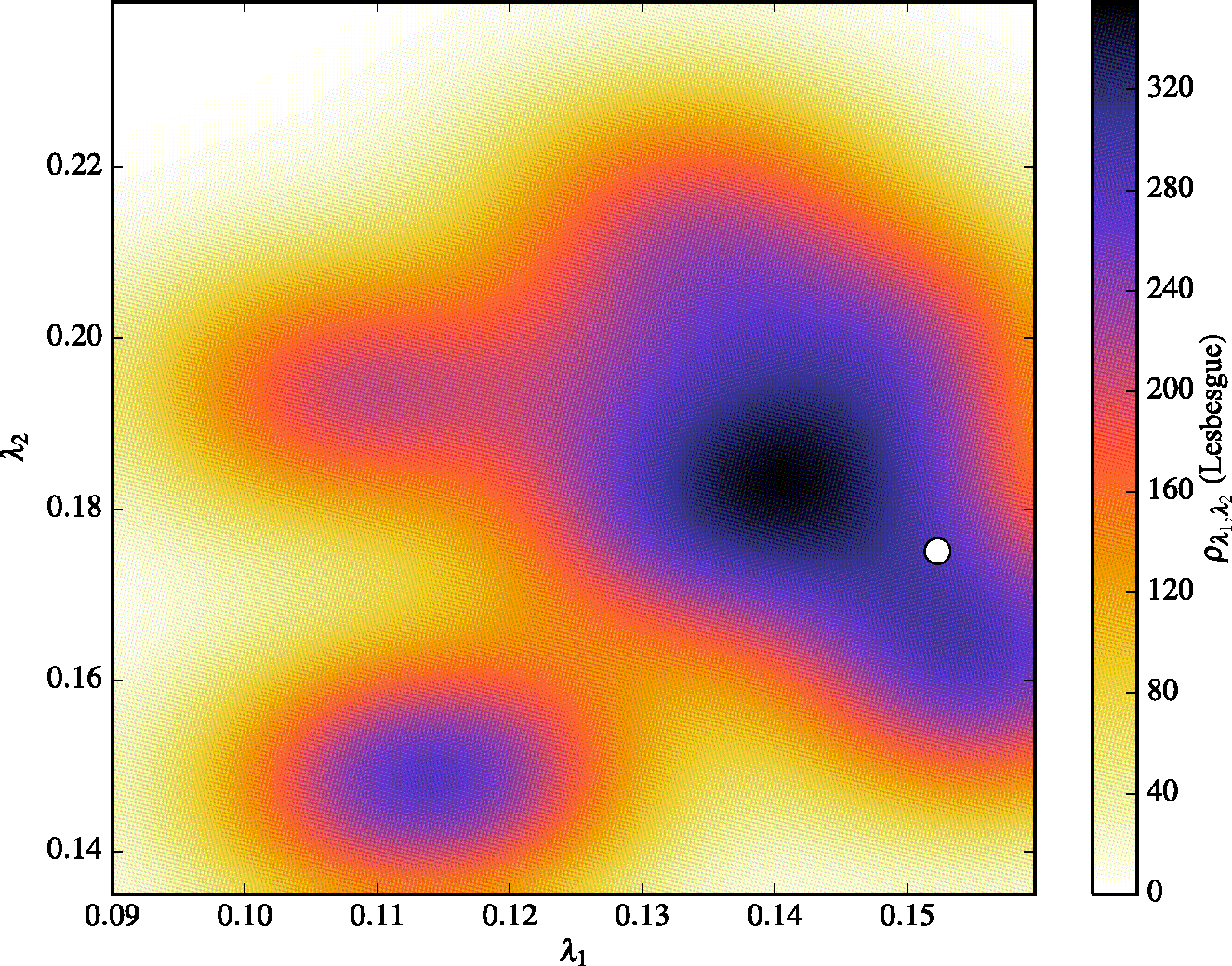}
        		\includegraphics[width=.24\textwidth]{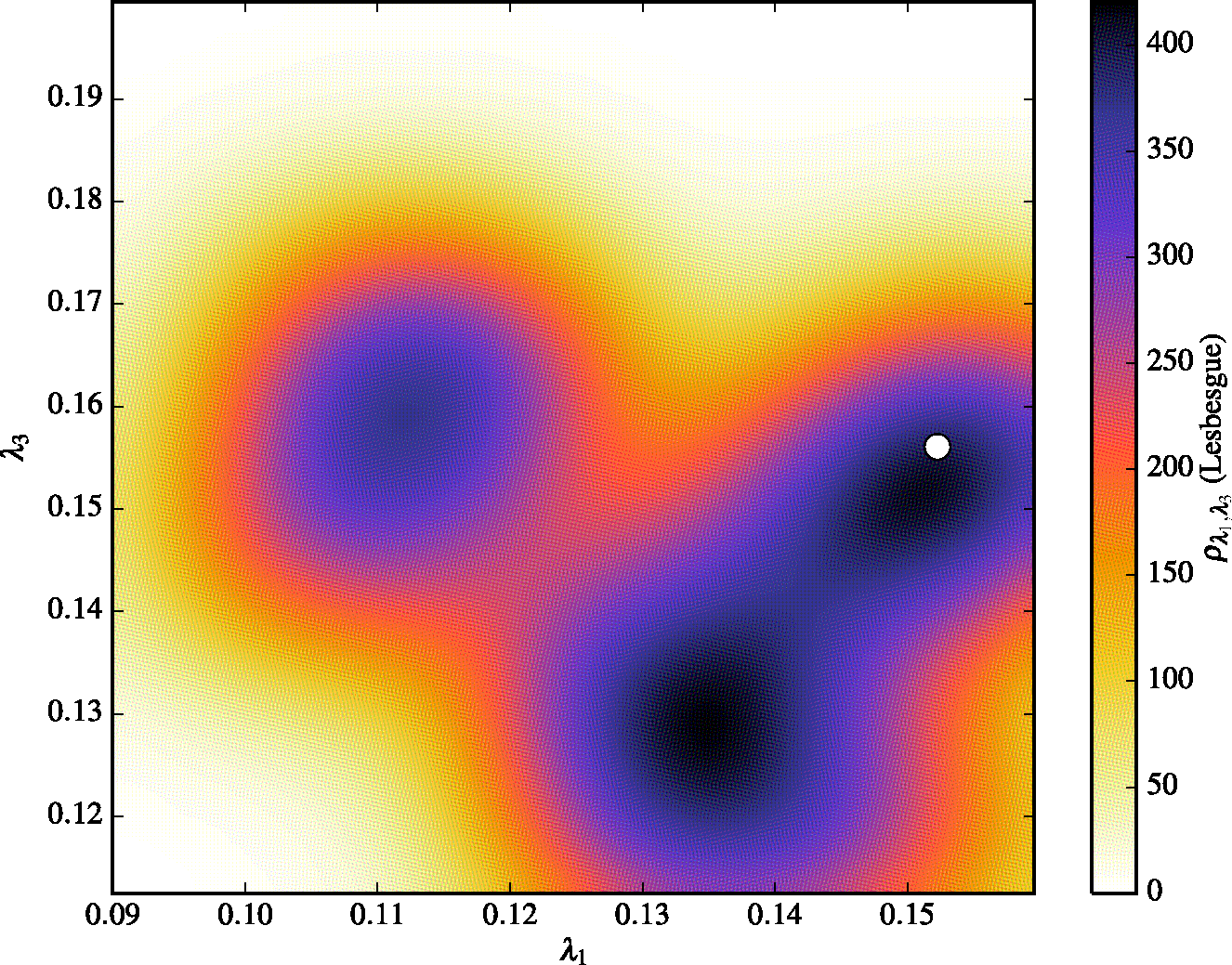}
        			\includegraphics[width=.24\textwidth]{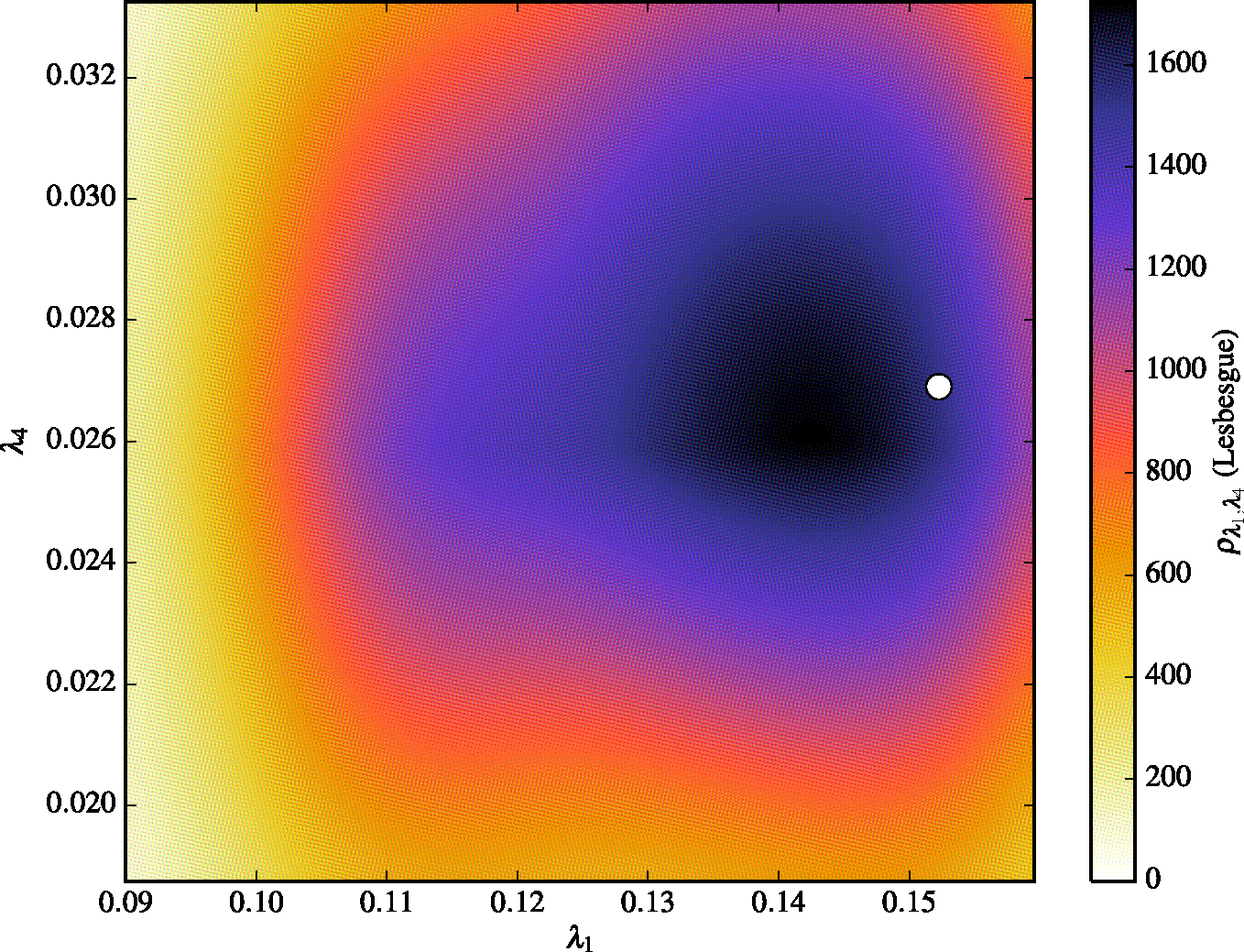} \\
		\includegraphics[width=.24\textwidth]{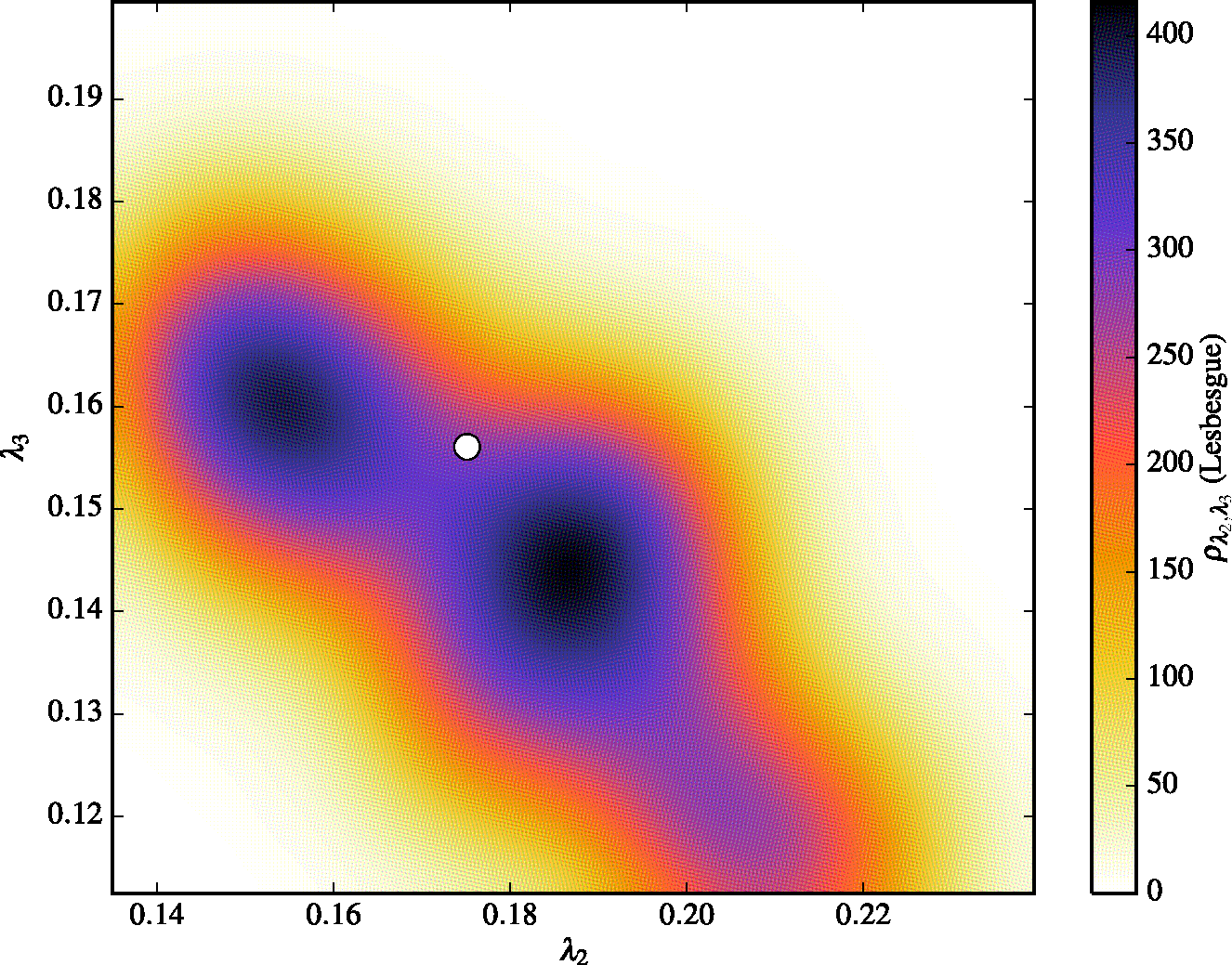}
        		\includegraphics[width=.24\textwidth]{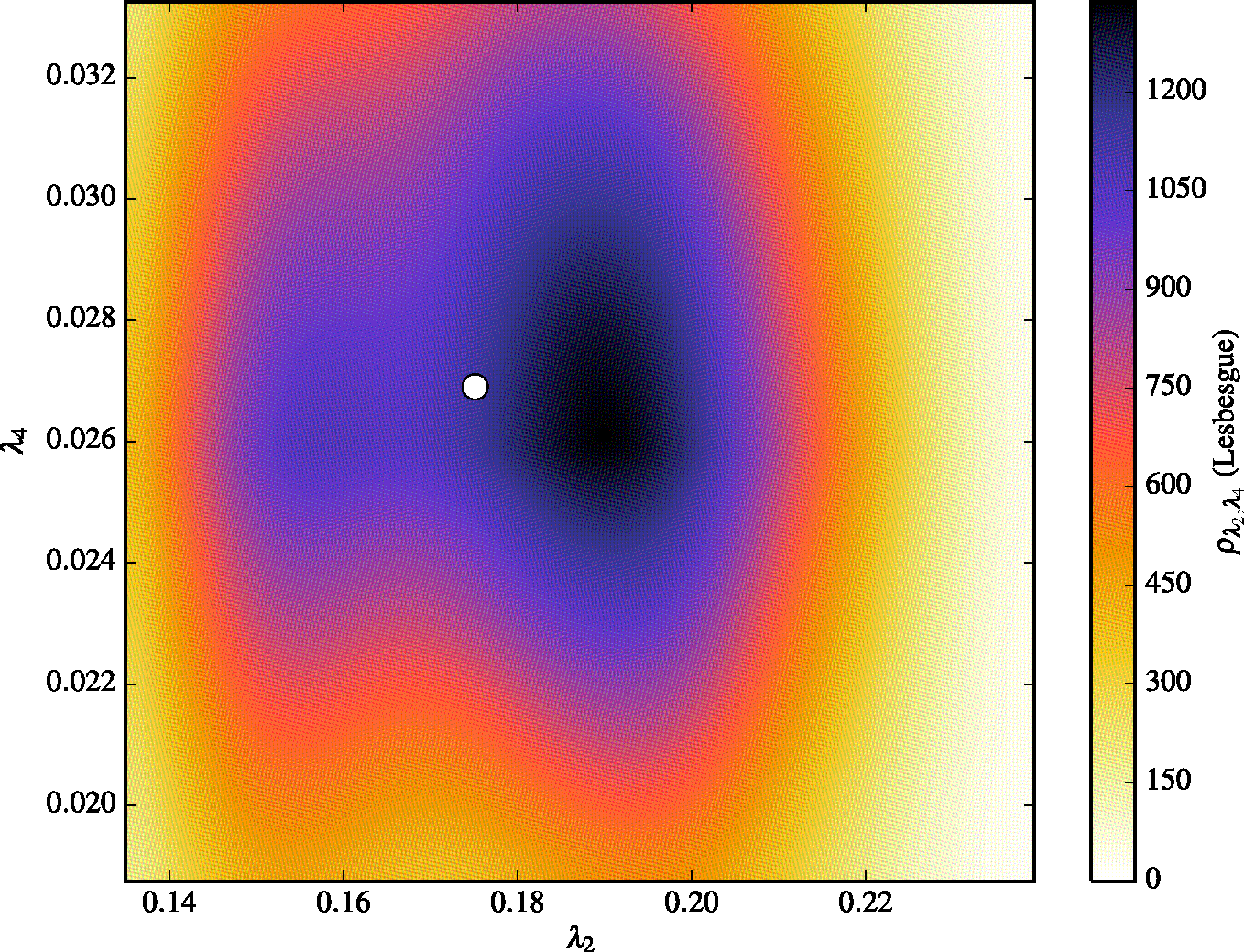}
				\includegraphics[width=.24\textwidth]{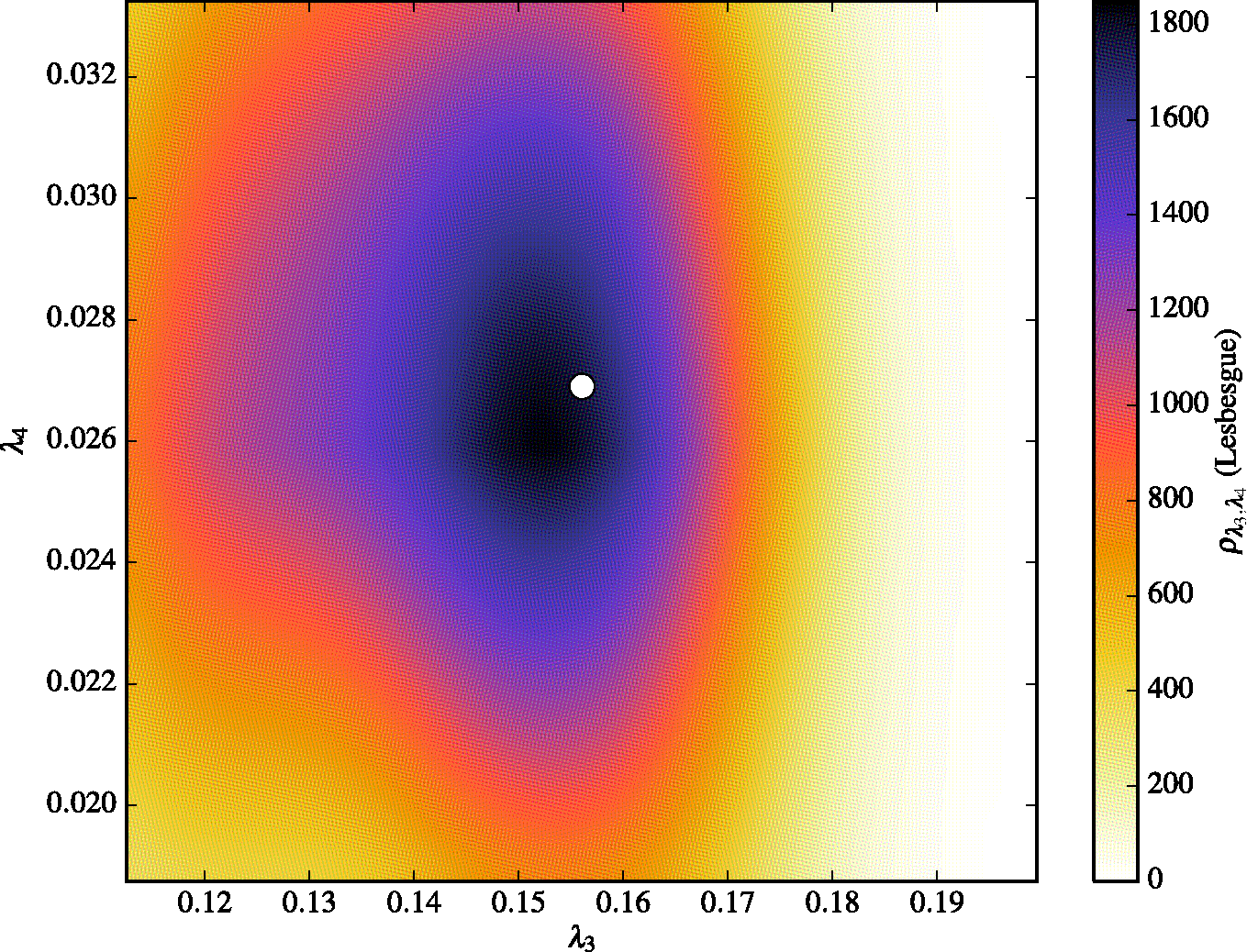}
        \caption{\it Plots of the marginals of the inverse solution using uniform samples for suboptimal stations [16, 40, 22, 88]. Here, $\rho_{\mathcal{D}}$ is defined as a uniform density on a small rectangular box centered at the reference QoI values associated with $\lambda_{ref} = (0.1523, 0.1751, 0.1561, 0.0269)$. The reference value is illustrated by a white circle. (left): In order from top to bottom $(\lambda_1, \lambda_2),(\lambda_1, \lambda_4),(\lambda_2, \lambda_4)$. (right): In order from top to bottom $(\lambda_1, \lambda_3),(\lambda_2, \lambda_3),(\lambda_3, \lambda_4)$.}\label{Fig:adcirc_bad}
    \end{center}
\end{figure}

\begin{remark}
	The results from this section are from Lindley Graham's PhD Thesis \cite{Lindleysthesis} and were computed during the summer of 2015.  The algorithm used to determine the optimal set of QoI has since developed into the methods described in Sections~\ref{Sec:Precision_Accuracy} and \ref{Sec:Optimizing}.  Although these results were found using a less tuned algorithm, they still display fundamental characteristics of inverse solutions defined by optimal and suboptimal sets of QoI.
\end{remark}

In Figure~\ref{Fig:subdomain} we see three plots of the subdomain.  The points in each plot represent particular stations of interest, the smooth coloring of the domain represents the difference in the computationally approximated maximum storm surge elevation over 480 uniform samples in $\Lambda$.  Notice in the top image we do not place potential station locations in regions of the subdomain that do not display significant sensitivity in storm surge elevations as the Manning's n coefficients (parameters) vary.

In Figure~\ref{Fig:adcirc_optimal} we see the marginal densities computed from the inverse solution for the optimal set of QoI seen in the bottom left of Figure~\ref{Fig:subdomain}.  
In Figure~\ref{Fig:adcirc_bad} we see the marginal densities computed from the inverse solution for the suboptimal set of QoIs seen in the bottom right of Figure~\ref{Fig:subdomain}.  
To clearly draw distinctions between these separate solutions, we examine the top left plots in Figures~\ref{Fig:adcirc_optimal} and \ref{Fig:adcirc_bad}.  Note that although both marginal densities contain the reference point $\lambda_{ref}$ within their supports, the solution from the optimal QoI does a much better job of concentrating high probability in a smaller region of $\Lambda$ containing this reference point.  
The relative measure of the support of the density computed for the optimal set of stations is $\mu_\Lambda((Q^{(opt)})^{-1}(B)) / \mu_\Lambda(\Lambda)=7.904E-3$, where as for the suboptimal set $\mu_\Lambda((Q^{(subopt)})^{-1}(B)) / \mu_\Lambda(\Lambda)=1.917E-2$.  Using the optimal set of stations has reduced the relative measure of the support of the inverse density by approximately a factor of 3.

\section{\bf Conclusion}\label{Sec:Conclusion}

In \cite{Butler2015b} the local skewness of the inverse image of a generalized rectangle in some $\mathcal{D}$ was defined.  
We have extended this initial work to quantify the global effect of skewness which is a way of quantifying the accuracy we can obtain in solution to the stochastic inverse problem from a finite number of samples.  
We also developed a way to quantify the precision in solutions to stochastic inverse problems in terms of measures of supports of densities and/or high probability events.  
A multicriteria optimization problem was defined to determine the best QoI map from the space of theoretically possible QoI that simultaneously reduces both the skewness and measure of the support of the inverse density.  
Several numerical examples demonstrated the effect of the QoI choice on the solution to the stochastic inverse problem. 

\vskip 10pt
{\it Future Work}
\vskip 10pt

As mentioned in a previous remark, the choice of weights used in Eq.~\eqref{Eq:dist_suppskew} is a current topic of discussion.  
It is clear as more computational resources are available we desire to reduce the $\mu_\Lambda$-measure of the inverse image more than the skewness, i.e., we set $\omega\ll 1$. 
A fundamental problem is to determine if the value of $\omega$ can be solved for directly as a function of the number of model solves available.

In this work we considered the discrete optimization problem, i.e., given $m$ sensors and $d$ theoretical QoI to choose from, determine the optimal set to improve solutions to the corresponding stochastic inverse problem.  
The continuous analogue is simply choose $m$ optimal QoI from infinitely many possibilities.  
As a concrete example, suppose that for the PDE in Section~\ref{Sec:heatplate} we can place two temperature sensors at {\em any} point in space-time, and we wish to choose the best two points in space-time to record data.
This problem is approached by considering the function $R : \mathcal{Q}\goto Y_\omega$, where $\mathcal{Q}$ is now the set of infinitely many possible sets of QoI, and finding local minima of this function.  
In the case of QoI being represented by some specific type of functional at a point in space time (temperature measurement), the space $\mathcal{Q}$ has dimension $(\dim(\Omega) + 1) \times m$ where $\Omega$ is the spatial domain of the model and $m$ is the number of QoI to be chosen.  Although the dimension space $\mathcal{Q}$ increases quickly as $m$ increase, there appears to be symmetry to exploit in this space.

In high dimensional parameter spaces any set implicitly defined by the solution to a stochastic inverse problem, skewed or not, is difficult to approximate with a reasonably low number of samples.  
In \cite{Adaptive} adaptive sampling algorithms are used to combat this curse of dimensionality.  With the ability to place samples near the boundary of sets in $\Lambda$ the error in the approximation of these sets is greatly reduced.  With the addition of gradient information the efficiency of these adaptive sampling algorithms is improved.  Preliminary results indicate the efficiency is improved further as the skewness of the set is reduced.  A thorough numerical exploration of the effects of skewness on the efficiency of adaptive sampling algorithms is an obvious path to explore.

In many driving applications the collection of field data is difficult, time consuming, and expensive.  Obtaining a good solution to the inverse problem while using as little data as possible is greatly desired.  In the examples in Section~\ref{Sec:Map_Defined} we look to determine the best set of $m$ QoI to use to solve the inverse problem, however, possibly a set of $m-1$ QoI produces a very similar inverse solution, or a set of $m-2$.  This information can greatly reduce the resources needed to gather the field data required.  With the $M_{Q^{(z)}}$ and $S_{Q^{(z)}}$ defined as in Section~\ref{Sec:Precision_Accuracy}, can we use $M_{Q^{(z),k}}$ and $S_{Q^{(z),k}}$ (the $\mu_\Lambda$-measure and skewness of the map $Q^{(z)}$ without the $k^{th}$ QoI) to draw conclusions about the maximum number of useful QoI to use from a given set?

\pagebreak
\bibliographystyle{plain}
\bibliography{masters_report.v4.arxiv}

\end{document}